\documentclass[12pt,twoside]{amsart}

\usepackage{amsmath,amsthm,amscd,amssymb,mathrsfs,graphicx,amsfonts,mathrsfs}
\usepackage{amsfonts}
\usepackage{amssymb,enumerate}
\usepackage{amsthm}
\usepackage[all]{xy}
\usepackage{hyperref}
\allowdisplaybreaks[4] \footskip=12pt
\renewcommand{\uppercasenonmath}[1]{}

\numberwithin{equation}{section} \theoremstyle{plain}
\newtheorem*{thm*}{Main Theorem}
\newtheorem{thm}{Theorem}[section]
\newtheorem{cor}[thm]{Corollary}
\newtheorem*{cor*}{Corollary}
\newtheorem{lem}[thm]{Lemma}
\newtheorem*{lem*}{Lemma}

\newtheorem*{fact*}{Fact}

\newtheorem*{nota*}{Notation}
\newtheorem{prop}[thm]{Proposition}
\newtheorem*{prop*}{Proposition}
\newtheorem{rem}[thm]{Remark}
\newtheorem*{rem*}{Remark}

\newtheorem*{observation*}{Observation}
\newtheorem{exa}[thm]{Example}
\newtheorem*{exa*}{Example}
\newtheorem{df}[thm]{Definition}
\newtheorem*{df*}{Definition}
\newtheorem{con}[thm]{Construction}
\newtheorem*{con*}{Construction}

\renewcommand{\geq}{\geqslant}
\renewcommand{\leq}{\leqslant}

\begin{document}
\begin{center}
{\large  \bf  Cohomology of torsion and completion of $N$-complexes}

\vspace{0.5cm} Xiaoyan Yang\\
Department of Mathematics, Northwest Normal University, Lanzhou 730070,
China
E-mail: yangxy@nwnu.edu.cn
\end{center}

\bigskip
\centerline { \bf  Abstract}
\leftskip10truemm \rightskip10truemm \noindent We introduce the notions of Koszul $N$-complex, $\check{\mathrm{C}}$ech $N$-complex and telescope $N$-complex, explicit derived torsion and derived completion functors in the derived category $\mathbf{D}_N(R)$ of $N$-complexes using the $\check{\mathrm{C}}$ech $N$-complex and the telescope $N$-complex. Moreover, we give an equivalence between the category of cohomologically $\mathfrak{a}$-torsion $N$-complexes and the category of cohomologically $\mathfrak{a}$-adic complete $N$-complexes,
and prove that over a commutative noetherian ring, via Koszul cohomology, via RHom cohomology (resp. $\otimes$ cohomology) and via local cohomology (resp. derived completion), all yield the same invariant.\\
\vbox to 0.3cm{}\\
{\it Key Words:} Koszul $N$-complex; telescope $N$-complex; torsion; completion\\
{\it 2010 Mathematics Subject Classification:} 13D07; 13D30; 13B35.

\leftskip0truemm \rightskip0truemm
\bigskip
\section* { \bf Introduction and Preliminaries}

The notion of $N$-complexes (graded objects with $N$-differentials $d$) was introduced by Mayer \cite{M} in his study of simplicial complexes and
its abstract framework of homological theory was studied by Kapranov \cite{K} and Dubois-Violette \cite{DV}. Since then the homological properties of $N$-complexes have attracted many authors, for example \cite{BHN,E,G12,GH,T,YD,YW}. Iyama, Kato and Miyachi \cite{IKM} studied the homotopy category $\mathbf{K}_N(\mathcal{B})$ of $N$-complexes of an
additive category $\mathcal{B}$ as well as the derived category $\mathbf{D}_N(\mathcal{A})$ of an abelian category $\mathcal{A}$. They proved that both $\mathbf{K}_N(\mathcal{B})$ and $\mathbf{D}_N(\mathcal{A})$ are triangulated, and established a theory of projective
(resp. injective) resolutions and derived functors. They also
showed that the well known equivalences between homotopy category of chain
complexes and their derived categories also generalize to the case of $N$-complexes.

Let $R$ be a commutative ring and $\mathfrak{a}$ an ideal of $R$. Denote by $\mathrm{Mod}R$ the category of $R$-modules. There are two operations associated to this data: the $\mathfrak{a}$-torsion and the $\mathfrak{a}$-adic completion. For an $R$-module $M$, the $\mathfrak{a}$-torsion elements form the $\mathfrak{a}$-torsion submodule $\Gamma_\mathfrak{a}(M)\cong\underrightarrow{\textrm{lim}}_{i>0}\mathrm{Hom}_R(R/\mathfrak{a}^i,M)$ of $M$. The $\mathfrak{a}$-adic completion of $M$ is $\Lambda_\mathfrak{a}(M):=\underleftarrow{\textrm{lim}}_{i>0}(R/\mathfrak{a}^i\otimes_RM)$.
Therefore, we have two additive functors
\begin{center}$\Gamma_\mathfrak{a},\Lambda_\mathfrak{a}:\mathrm{Mod}R\rightarrow\mathrm{Mod}R$.\end{center}
The derived category of $\mathrm{Mod}R$ is denoted by $\mathbf{D}(R)$. Then the derived functors
\begin{center}$\mathrm{R}\Gamma_\mathfrak{a},\mathrm{L}\Lambda_\mathfrak{a}:
\mathbf{D}(R)\rightarrow\mathbf{D}(R)$\end{center}exist. The right derived functor $\mathrm{R}\Gamma_\mathfrak{a}$ has been studied in great length already by
Grothendieck and others in the context of local cohomology. The left derived functors $\mathrm{L}\Lambda_\mathfrak{a}$ was studied by Matlis \cite{M} and Greenlees-May
\cite{GM}.

Let $\mathfrak{a}$ be a weakly proregular ideal of $R$,
this includes the noetherian case, but there are other interesting examples. Porta, Shaul and Yekutieli \cite{PSY} extended earlier work by Alonso-Jeremias-Lipman \cite{AJL}, Schenzel \cite{S} and
Dwyer-Greenlees \cite{DG}. They proved that the derived functors $\mathrm{R}\Gamma_\mathfrak{a}$ and $\mathrm{L}\Lambda_\mathfrak{a}$ can be computed by telescope complexes, and established the MGM equivalence, where the letters ``MGM'' stand for Matlis, Greenlees and May.

The first aim of this paper is to extend works of Porta, Shaul and Yekutieli to the category of $N$-complexes. We introduce the definitions of Koszul $N$-complex, $\check{\mathrm{C}}$ech $N$-complex and telescope $N$-complex, and explicit derived torsion and derived completion functors in $\mathbf{D}_N(R)$ using these $N$-complexes.

\vspace{2mm} \noindent{\bf Theorem A.}\label{Th1.4} {\it{Let $\textbf{x}=x_1,\cdots,x_d$ be a weakly proregular sequence in $R$ and $\mathfrak{a}$ the ideal generated by $\textbf{x}$. For any $N$-complex $X$, there are functorial quasi-isomorphisms \begin{center}$\mathrm{R}\Gamma_{\mathfrak{a}}(X)\stackrel{\simeq}\rightarrow \check{C}(\textbf{x};R)\otimes_RX\cong\mathrm{Tel}(\textbf{x};R)\otimes_RX$,\end{center} \begin{center}$\mathrm{Hom}_R(\mathrm{Tel}(\textbf{x};R),X)\stackrel{\simeq}\rightarrow
\mathrm{L}\Lambda_{\mathfrak{a}}(X)$.\end{center}}}
\vspace{2mm}

 Denote by  $\mathbf{D}_N(R)_{\mathfrak{a}\textrm{-tor}}$ and $\mathbf{D}_N(R)_{\mathfrak{a}\textrm{-com}}$ the full subcategories of $\mathbf{D}_N(R)$ consisting of cohomologically $\mathfrak{a}$-torsion
$N$-complexes and cohomologically $\mathfrak{a}$-adic
complete $N$-complexes, respectively (see Definition \ref{lem:5.1}). We show that

\vspace{2mm} \noindent{\bf Theorem B.}\label{Th1.4} {\it{Let $\mathfrak{a}$ be a weakly proregular ideal of $R$. Then the functors \begin{center}$\mathrm{R}\Gamma_{\mathfrak{a}}:\mathbf{D}_N(R)_{\mathfrak{a}\textrm{-}\mathrm{com}}
\rightleftarrows\mathbf{D}_N(R)_{\mathfrak{a}\textrm{-}\mathrm{tor}}:\mathrm{L}\Lambda_{\mathfrak{a}}$\end{center} form an equivalence.}}
\vspace{2mm}

Let $\mathfrak{a}$ be an ideal in a commutative noetherian ring $R$ and $K$ the Koszul complex on
a finite set of $n$ generators for $\mathfrak{a}$. It is well known that the following numbers are
equal when $M$ is a finitely generated $R$-module:

\quad$\bullet$ $n+\mathrm{inf}\{\ell\in\mathbb{Z}\hspace{0.03cm}|\hspace{0.03cm}\mathrm{H}_\ell(K\otimes_RM)\neq0\}$;

\quad$\bullet$ $\mathrm{inf}\{\ell\in\mathbb{Z}\hspace{0.03cm}|\hspace{0.03cm}\mathrm{Ext}^\ell_R(R/\mathfrak{a},M)\neq0\}$;

\quad$\bullet$ $\mathrm{inf}\{\ell\in\mathbb{Z}\hspace{0.03cm}|\hspace{0.03cm}\mathrm{H}^{\ell}_\mathfrak{a}(M)\cong0\}$, where $\mathrm{H}^{\ast}_\mathfrak{a}(M)$ is the $\mathfrak{a}$-local cohomology of $M$.\\
Each of the
quantities displayed above is meaningful. These have proved to be of immense utility even in dealing with
problems concerning modules alone. Foxby and Iyengar \cite{FI} proved that the numbers obtained from the
three formulas above coincide for any complex. It is natural to ask if these three approaches yield the same invariant for $N$-complexes.
The second aim of
current paper is to answer the question for any $N$-complex and consider its dual statement over commutative noetherian rings.

\section{\bf Preliminaries and basic facts}
We assume throughout this paper that all rings are commutative.

This section is devoted to recalling some notions and basic facts which we need in
the later sections.
For terminology we shall follow \cite{BHN,IKM} and \cite{YD}.

\vspace{2mm}
{\bf $N$-complexes.} Fix an integer $N\geq 2$.
An $N$-complex $X$ is a sequence of $R$-modules \begin{center}$\cdots\stackrel{d^{n-2}}\longrightarrow
X^{n-1}\stackrel{d^{n-1}}\longrightarrow
X^n\stackrel{d^{n}}\longrightarrow
X^{n+1}\stackrel{d^{n+1}}\longrightarrow\cdots$
\end{center} satisfying $d^N=0$. That is, composing any $N$-consecutive morphisms gives $0$. A morphism $f:X\rightarrow Y$
of $N$-complexes is a collection of maps $f^n:X^n\rightarrow Y^n$ making all the rectangles
commute. In this way we get a category of $N$-complexes, denoted by $\mathbf{C}_N(R)$.

For any $R$-module $M$, $j\in\mathbb{ Z}$ and $t=1,\cdots,N$, we define
\begin{center}$D^j_t(M):\cdots0\rightarrow X^{j-t+1}\xrightarrow{d^{j-t+1}}\cdots\xrightarrow{d^{j-2}}X^{j-1}\xrightarrow{d^{j-1}}X^j
\rightarrow0\rightarrow\cdots$\end{center}
be an $N$-complex given by $X^n=M$ for all $j-t
+1\leq n\leq j$ and $d^n=1_M$ for all $j-t
+1\leq n\leq j-1$.

Let $X$ be an $N$-complex.
For $n\in\mathbb{Z}$, we define \begin{center}$\textrm{Z}^n_t(X)=\textrm{Ker}(d^{n+t-1}\cdots d^n)$, $\textrm{B}^n_t(X)=\textrm{Im}(d^{n-1}\cdots d^{n-t})$ for $t=0,\cdots,N$,\end{center}
\begin{center}$\textrm{C}^n_t(X)=\textrm{Coker}(d^{n-1}\cdots d^{n-t})$, $\textrm{H}^n_t(X)=\textrm{Z}^n_t(X)/\textrm{B}^n_{N-t}(X)$ for $t=1,\cdots,N-1$.\end{center}
An $N$-complex $X$ is called
$N$-acyclic if $\textrm{H}^n_t(X)=0$ for all $n$ and $t$.

\begin{prop} \label{prop:0.1}{\rm(\cite{IKM})} {\it{Let $0\rightarrow X\rightarrow Y\rightarrow Z\rightarrow 0$ be a short exact sequence in $\mathbf{C}_N(R)$. For $n\in\mathbb{Z}$ and $1\leq t\leq N-1$, there is a long exact sequence of cohomologies \begin{center}$\cdots \rightarrow\mathrm{H}^{n-(N-t)}_{N-t}(Z)\rightarrow\mathrm{H}^{n}_t(X)\rightarrow\mathrm{H}^{n}_t(Y)
\rightarrow\mathrm{H}^{n}_t(Z)\rightarrow\mathrm{H}^{n+t}_{N-t}(X)\rightarrow\cdots$.\end{center}}}
\end{prop}

Let $X$ be an $N$-complex. Define suspension functors $\Sigma,\Sigma^{-1}:\mathbf{K}_N(R)\rightarrow\mathbf{K}_N(R)$ as follows\begin{center}$(\Sigma X)^n=X^{n+1}\oplus\cdots\oplus X^{n+N-1}$,  $d_{\Sigma X}=\left[\begin{smallmatrix} 0&1&0&\cdots&0&0\\ 0&0&1&\cdots&0&0\\\vdots&\vdots&\vdots&\vdots&\vdots&\vdots\\ 0&0&0&\cdots&0&1\\-d^{N-1}&-d^{N-2}&-d^{N-3}&\cdots&-d^2&-d \end{smallmatrix}\right]$,\end{center}
\begin{center}$(\Sigma^{-1}X)^n=X^{n-N+1}\oplus\cdots\oplus X^{n-1}$,  $d_{\Sigma^{-1}X}=\left[\begin{smallmatrix} -d&1&0&\cdots&0&0\\
-d^2&0&1&\cdots&0&0\\\vdots&\vdots&\vdots&\vdots&\vdots&\vdots\\
-d^{N-2}&0&0&\cdots&0&1\\-d^{N-1}&0&0&\cdots&0&0
\end{smallmatrix}\right]$.\end{center}
Let $f:X\rightarrow Y$ be a morphism in $\mathbf{C}_N(R)$. The mapping cone $C(f)$ of $f$ is defined as
\begin{center}$C(f)^n=Y^n\oplus(\Sigma X)^n$, $d^n_{C(f)}=\left[\begin{smallmatrix} d&f&0&\cdots&0&0\\
0&0&1&\cdots&0&0\\\vdots&\vdots&\vdots&\vdots&\vdots&\vdots\\
0&0&0&\cdots&0&1\\0&-d^{N-1}&-d^{N-2}&\cdots&-d^2&-d
\end{smallmatrix}\right]$.\end{center}

Two morphisms $f,g:X\rightarrow Y$ of $N$-complexes are called
homotopic if there exists $\{s^n:X^n\rightarrow Y^{n-N+1}\}$ such that
\begin{center}$g^n-f^n=\sum^{N-1}_{i=0}d^{N-1-i}s^{n+i}d^i,\
\forall\ n$.\end{center}
We denote the homotopy category of $N$-complexes by
$\mathbf{K}_N(R)$. Then the category $\mathbf{K}_N(R)$ is triangulated, and every exact triangle in $\mathbf{K}_N(R)$ is isomorphic to the form
\begin{center}$X\xrightarrow{f}Y\xrightarrow{g}C(f)\xrightarrow{h}\Sigma X$,\end{center}where $X,Y\in\mathbf{K}_N(R)$ and $g=\left[\begin{smallmatrix}
1\\0\\\vdots\\0\end{smallmatrix}\right],\
h=\left[\begin{smallmatrix} 0&1&0&\cdots&0&0\\
0&0&1&\cdots&0&0\\
\vdots&\vdots&\vdots&\vdots&\vdots&\vdots\\0&0&0&\cdots&0&1
\end{smallmatrix}\right]$.

A morphism $f:X\rightarrow Y$ is called a quasi-isomorphism if the induced morphism $\mathrm{H}^i_t(f):\mathrm{H}^i_t(X)\rightarrow\mathrm{H}^i_t(Y)$ is an isomorphism for any $i$ and $t=1,\cdots,N-1$, or equivalently if the mapping cone
$C(f)$ belongs to $\mathbf{K}^{\mathrm{ac}}_N(R)$ the full subcategory of $\mathbf{K}_N(R)$ consisting of $N$-acyclic $N$-complexes. The derived category $\mathbf{D}_N(R)$ of $N$-complexes is defined as the quotient category
$\mathbf{K}_N(R)/\mathbf{K}^{\mathrm{ac}}_N(R)$, which is also triangulated.

\begin{df}\label{lem:0.2} {\rm (\cite{K}) Let $q$ be a primitive $N$-th root of 1 ($q^N=1$), and let $(X,d_X)$,$(Y,d_Y)$ be two $N$-complexes of $R$-modules.

(a) The $q$-Hom is the $N$-complex $\mathrm{Hom}_R(X,Y)$ defined by
\begin{center}$\mathrm{Hom}_R(X,Y)^n=\prod_{i\in\mathbb{Z}}\mathrm{Hom}_R(X^i,Y^{i+n})$ with differential $d^n(f^i)=d^{i+n}_Yf^i-q^nf^{i+1}d^i_X$.\end{center}

(b) The $q$-tensor product is the $N$-complex $X\otimes_RY$ defined by
\begin{center}$(X\otimes_RY)^n=\coprod_{i\in\mathbb{Z}}(X^i\otimes_RY^{n-i})$ with differential $d^n(x\otimes y)=d_X(x)\otimes y+q^{|x|}x\otimes d_Y(y)$,\end{center}where $x,y$ are supposed to be homogeneous and $|x|$ denotes the degree of $x$.}
\end{df}

\begin{rem}\label{lem:0.4}{\rm (1) By the definition of $q$-Hom and $q$-tensor product of $N$-complexes and the isomorphism in \cite[Theorem 2.4]{IKM}, one can check the following isomorphisms:

\begin{center}$\mathrm{Hom}_R(X,\Sigma Y)\cong\Sigma\mathrm{Hom}_R(X,Y)$,\end{center}

\begin{center}$\mathrm{Hom}_R(\Sigma X,Y)\cong \mathrm{Hom}_R(X,\Sigma^{-1}Y)\cong\Sigma^{-1}\mathrm{Hom}_R(X,Y)$,\end{center}

\begin{center}$\Sigma X\otimes_RY\cong X\otimes_R\Sigma Y\cong\Sigma(X\otimes_RY)$.\end{center}

(2) It follows from \cite[Corollary 4.5]{Mi} that $(X\otimes_R-,\mathrm{Hom}_R(X,-))$ form a adjoint pair.}
\end{rem}

\begin{lem}\label{lem:0.3}{\it{For any morphism $f:X\rightarrow Y$ in $\mathbf{C}_N(R)$ and any $N$-complex $Z$, one has

$\mathrm{(1)}$ $\mathrm{Hom}_R(Z,C(f))\cong C(\mathrm{Hom}_R(Z,f))$.

$\mathrm{(2)}$  $\mathrm{Hom}_R(C(f),Z)\cong \Sigma^{-1}C(\mathrm{Hom}_R(f,Z))$.

$\mathrm{(3)}$ $C(f)\otimes_R Z\cong C(f\otimes Z)$ and $Z\otimes_RC(f)\cong C(Z\otimes f)$.}}
\end{lem}

\begin{proof} (1) For any $\left[\begin{smallmatrix} \alpha^i \\
\beta^i\end{smallmatrix}\right]\in\mathrm{Hom}_R(Z,C(f))^n\cong\mathrm{Hom}_R(Z,Y)^n\oplus
\mathrm{Hom}_R(Z,\Sigma X)^n$, we have
\begin{center}$\begin{aligned}d^n_{\mathrm{Hom}_R(Z,C(f))}\left[\begin{smallmatrix} \alpha^i \\
\beta^i\end{smallmatrix}\right]
&=d^{n+i}_{C(f)}\left[\begin{smallmatrix} \alpha^i \\
\beta^i\end{smallmatrix}\right]-q^n\left[\begin{smallmatrix} \alpha^{i+1} \\
\beta^{i+1}\end{smallmatrix}\right]d^i_Z\\
&=\left[\begin{smallmatrix} d^{n+i}_Y& f^{n+i+1}\\
0&d^{n+i}_{\Sigma X}\end{smallmatrix}\right]\left[\begin{smallmatrix} \alpha^i \\
\beta^i\end{smallmatrix}\right]-q^n\left[\begin{smallmatrix} \alpha^{i+1}d^i_Z \\
\beta^{i+1}d^i_Z\end{smallmatrix}\right]\\
&=\left[\begin{smallmatrix} d^{n+i}_Y\alpha^i+ f^{n+i+1}\beta^i-q^n\alpha^{i+1}d^i_Z \\
d^{n+i}_{\Sigma X}\beta^i-q^n\beta^{i+1}d^i_Z\end{smallmatrix}\right]\\
&=\left[\begin{smallmatrix} d^{n}_{\mathrm{Hom}_R(Z,Y)}& \mathrm{Hom}_R(Z,f)^{n+1}\\
0&d^{n}_{\mathrm{Hom}_R(Z,\Sigma X)}\end{smallmatrix}\right]\left[\begin{smallmatrix} \alpha^i \\
\beta^i\end{smallmatrix}\right],\end{aligned}$\end{center}which implies that $d^n_{\mathrm{Hom}_R(Z,C(f))}=\left[\begin{smallmatrix} d^{n}_{\mathrm{Hom}_R(Z,Y)}& \mathrm{Hom}_R(Z,f)^{n+1}\\
0&d^{n}_{\Sigma\mathrm{Hom}_R(Z,X)}\end{smallmatrix}\right]$, as desired.

(2) For any $\left[\begin{smallmatrix} \alpha^i \\
\beta^i\end{smallmatrix}\right]\in\mathrm{Hom}_R(C(f),Z)^n\cong\mathrm{Hom}_R(Y,Z)^n\oplus
\mathrm{Hom}_R(\Sigma X,Z)^n$, since $\left[\begin{smallmatrix} \alpha^i \\
\beta^i\end{smallmatrix}\right]$ corresponds to a morphism $\left[\begin{smallmatrix} \alpha^i &
\beta^i\end{smallmatrix}\right]:Y^i\oplus(\Sigma X)^i\rightarrow Z^{n+i}$ we have
\begin{center}$\begin{aligned}d^n_{\mathrm{Hom}_R(C(f),Z)}\left[\begin{smallmatrix} \alpha^i \\
\beta^i\end{smallmatrix}\right]
&=d^{n+i}_{Z}\left[\begin{smallmatrix} \alpha^i &
\beta^i\end{smallmatrix}\right]-q^n\left[\begin{smallmatrix} \alpha^{i+1} &
\beta^{i+1}\end{smallmatrix}\right]d^i_{C(f)}\\
&=d^{n+i}_Z\left[\begin{smallmatrix} \alpha^{i} &
\beta^{i}\end{smallmatrix}\right]-q^n\left[\begin{smallmatrix} \alpha^{i+1} &
\beta^{i+1}\end{smallmatrix}\right]\left[\begin{smallmatrix} d^{i}_Y& f^{i+1}\\
0&d^{i}_{\Sigma X}\end{smallmatrix}\right]\\
&=\left[\begin{smallmatrix} d^{n+i}_Z\alpha^i-q^n\alpha^{i+1}d^i_Y &
d^{n+i}_{Z}\beta^i-q^n\alpha^{i+1}f^{i+1}-q^n\beta^{i+1}d^i_{\Sigma X}\end{smallmatrix}\right]\\
&=\left[\begin{smallmatrix} d^{n}_{\mathrm{Hom}_R(Y,Z)}& 0\\
-q^n\mathrm{Hom}_R(f,Z)^{i+1}&d^{n}_{\mathrm{Hom}_R(\Sigma X,Z)}\end{smallmatrix}\right]\left[\begin{smallmatrix} \alpha^i \\
\beta^i\end{smallmatrix}\right],\end{aligned}$\end{center}which implies that $d^n_{\mathrm{Hom}_R(C(f),Z)}=\left[\begin{smallmatrix} d^{n}_{\mathrm{Hom}_R(Z,Y)}& 0\\
-q^n\mathrm{Hom}_R(f,Z)^{i+1}&d^{n}_{\Sigma\mathrm{Hom}_R(Z,X)}\end{smallmatrix}\right]$, as desired.

(3) These follow from $d_{C(f)\otimes_RZ}=\left[\begin{smallmatrix} d_{Y\otimes_RZ}& f\otimes Z\\
0&d_{\Sigma X\otimes_RZ}\end{smallmatrix}\right]$ and $d_{Z\otimes_RC(f)}=\left[\begin{smallmatrix} d_{Z\otimes_RY}& Z\otimes f\\
0&d_{Z\otimes_R\Sigma X}\end{smallmatrix}\right]$.
\end{proof}

Let $\mathscr{P}$ be the class of
projective $R$-modules. An $N$-complex $P$ is called semi-projective if $P^n\in\mathscr{P}$ for all $n$, and every $f:P\rightarrow E$ is null homotopic whenever $E\in\mathbf{K}^{\mathrm{ac}}_N(R)$. Let $\mathscr{I}$ be the class of injective $R$-modules. An $N$-complex $I$ is called semi-injective if $I^n\in\mathscr{I}$ for all $n$, and every $f:E\rightarrow I$ is null homotopic whenever $E\in\mathbf{K}^{\mathrm{ac}}_N(R)$.

Let $X$ be an $N$-complex. By Lemma \ref{lem:0.3}, we have four triangle functors $\mathrm{Hom}_R(X,-)$, $\mathrm{Hom}_R(-,X):\mathbf{K}_N(R)\rightarrow\mathbf{K}_N(\mathbb{Z})$ and $-\otimes_RX,X\otimes_R-:\mathbf{K}_N(R)\rightarrow\mathbf{K}_N(\mathbb{Z})$.
Then \cite[Corollary 3.29]{IKM} yields the following derived functors
\begin{center}$\mathrm{RHom}_R(X,-),\mathrm{RHom}_R(-,X):\mathbf{D}_N(R)\rightarrow\mathbf{D}_N(\mathbb{Z})$,\end{center}
\begin{center} $-\otimes^\mathrm{L}_RX,X\otimes^\mathrm{L}_R-:\mathbf{D}_N(R)\rightarrow\mathbf{D}_N(\mathbb{Z})$.\end{center} They can be computed via semi-projective and semi-injective resolution of the $N$-complexes by \cite[Theorem 3.27]{IKM}, respectively.

Let $X$ be an $N$-complex. The stupid truncation $\tau_{>i}(X):0\rightarrow X^{i+1}\rightarrow X^{i+2}\rightarrow\cdots$ and $\tau_{\leq i}(X):\cdots\rightarrow X^{i-1}\rightarrow X^{i}\rightarrow0$.
For $i,j\in\mathbb{Z}$ let $\mathbf{C}^{[i,j]}_N(R)$ be the full subcategory of $\mathbf{C}_N(R)$ whose objects
are the $N$-complexes concentrated in the degree range $[i,j]:=\{i,\cdots,j\}$. Here is a useful criterion for quasi-isomorphisms.

\begin{lem}\label{lem:3.3}{\it{Let $R$ and $R'$ be two rings and $F,G:\mathrm{Mod}R\rightarrow \mathbf{C}_N(R')$ two additive
functors, and let $\eta:F\rightarrow G$ be a natural transformation. Consider the extensions $F,G:\mathbf{C}_N(R)\rightarrow \mathbf{C}_N(R')$. Suppose $X\in\mathbf{C}_N(R)$
satisfies the following conditions:

$\mathrm{(1)}$ There are $j_0,j_1\in\mathbb{Z}$ such that $F(X^i),G(X^i)\in \mathbf{C}^{[j_0,j_1]}_N(R')$ for every $i\in\mathbb{Z}$.

$\mathrm{(2)}$  The homomorphism $\eta_{X^i}:F(X^i)\rightarrow G(X^i)$ is a quasi-isomorphism for every
$i\in\mathbb{Z}$.\\
Then $\eta_X:F(X)\rightarrow G(X)$ is a quasi-isomorphism.}}
\end{lem}
\begin{proof}  Assume that $X$ is bounded. If $X\simeq D^i_1(M)$ for some $R$-module $M$ and $i\in\mathbb{Z}$ then this is given. Otherwise the inductive step
is done using the short exact sequence \begin{center}$0\rightarrow \tau_{>i}(X)\rightarrow X\rightarrow\tau_{\leq i}(X)\rightarrow0$\end{center} of $N$-complexes.
Now assume $X$ is arbitrary. We prove that $\mathrm{H}^i_t(\eta_X):\mathrm{H}^i_t(F(X))\rightarrow\mathrm{H}^i_t(G(X))$ is an isomorphism for every $i\in \mathbb{Z}$ and a fixed $t$. For any $i\leq j$ set $\tau_{[i,j]}:=\tau_{\leq j}\circ\tau_{>i}$. Given an integer $i$, the morphism
$\mathrm{H}^i_t(\eta_X)$ only depends on the morphism
\begin{center}$\tau_{[i-N+t,i+t]}(\eta_X):\tau_{[i-N+t,i+t]}(F(X))\rightarrow\tau_{[i-N+t,i+t]}(G(X))$\end{center}of $N$-complexes.
Thus we can replace $\eta_X$ with $\eta_{X'}$, where $X'=\tau_{[j_0+i-N+t,j_1+i+t]}(X)$.
But $X'$ is bounded, so the morphism $\eta_{X'}$ is a quasi-isomorphism.
\end{proof}

\bigskip
\section{\bf The Koszul $N$-complex}
In this section, we give a construction of Koszul $N$-complexes and compute the cohomology of a few concrete Koszul $N$-complexes.

\begin{df}\label{lem:1.1}{\rm Let $x$ be an element in $R$. The Koszul $N$-complex on $x$,
denoted by $K^\bullet(x;R)$, is the mapping cone $C(x)$ of $x:D^0_1(R)\rightarrow D^0_1(R)$
\begin{center}$0\rightarrow R\stackrel{1}\rightarrow R\stackrel{1}\rightarrow\cdots\stackrel{1}\rightarrow R\stackrel{x}\rightarrow R\rightarrow0$\end{center}
with $R$ in degrees $-N+1,\cdots,0$. Suppose we are given a sequence $\emph{\textbf{x}}=x_1,\cdots,x_d$
of elements in $R$. By induction, the Koszul $N$-complex on $\emph{\textbf{x}}$, denoted by $K^\bullet(\emph{\textbf{x}};R)$, is the mapping cone $C(x_d)$ of $x_d:K^\bullet(x_1,\cdots,x_{d-1};R)\rightarrow K^\bullet(x_1,\cdots,x_{d-1};R)$.

One can check that $K^\bullet(\emph{\textbf{x}};R)\cong K^\bullet(x_1;R)\otimes_R\cdots\otimes_RK^\bullet(x_d;R)$.}
\end{df}

\begin{exa}\label{lem:1.2}{\rm Let $x,y,z$ be three elements in $R$.

(1) For $N=3$, the Koszul $3$-complex on $x$ is\begin{center}$K^\bullet(x;R):0\rightarrow R\stackrel{1}\rightarrow R\stackrel{x}\rightarrow R\rightarrow0$\end{center}with $R$ in degrees $-2,-1,0$. The Koszul $3$-complex on $x,y$ is \begin{center}$K^\bullet(x,y;R):0\rightarrow R\xrightarrow{\left[\begin{smallmatrix} 1\\ -1 \end{smallmatrix}\right]} R^2\xrightarrow{\left[\begin{smallmatrix}y&0\\0& 1\\-x&-x \end{smallmatrix}\right]} R^3\xrightarrow{\left[\begin{smallmatrix} 1&y&0\\ 0&0&1 \end{smallmatrix}\right]} R^2\xrightarrow{\left[\begin{smallmatrix} x&y \end{smallmatrix}\right]}R\xrightarrow{}0$\end{center}with the five nonzero modules in degrees $-4,\cdots,0$.
The Koszul $3$-complex on $x,y,z$ is \begin{center}$K^\bullet(x,y,z;R):0\rightarrow R\xrightarrow{\left[\begin{smallmatrix} 1\\ -1\\1 \end{smallmatrix}\right]} R^3\xrightarrow{\left[\begin{smallmatrix}z&0&0\\0& 1&0\\0&0&1\\-y&-y&0\\1&0&-1\\0&x&x \end{smallmatrix}\right]} R^6\xrightarrow{\left[\begin{smallmatrix} 1&z&0&0&0&0\\-1& 0&z&0&0&0\\0& 0&0&1&0&0 \\0& 0&0&0&1&0\\ 0& 0&0&0&0&1\\0& -y&-y&-1&-y&0\\0& x&x&0&0&-1\end{smallmatrix}\right]} R^7\xrightarrow{\left[\begin{smallmatrix} y&0&z&0&0&0&0\\0& 1&0&z&0&0&0\\-x&-x&0&0&z&0& 0\\0& 0&0&0&0&1&0\\ 0& 0&0&0&0&0&1\\0& 0&-x&-xy&-y&-x&-y\end{smallmatrix}\right]}R^6\xrightarrow{\left[\begin{smallmatrix} 1&y&0&z&0&0\\ 0&0&1&0&z&0\\0&0&0&0&0&1 \end{smallmatrix}\right]} R^3\xrightarrow{\left[\begin{smallmatrix} x&y &z\end{smallmatrix}\right]}R\xrightarrow{}0$\end{center}with the seven nonzero modules in degrees $-6,\cdots,0$.

(2) For $N=4$, the Koszul $4$-complex on $x$ is\begin{center}$K^\bullet(x;R):0\rightarrow R\stackrel{1}\rightarrow R\stackrel{1}\rightarrow R\stackrel{x}\rightarrow R\rightarrow0$\end{center}with $R$ in degrees $-3,-2,-1,0$. The Koszul $4$-complex on $x,y$ is \begin{center}$K^\bullet(x,y;R):0\rightarrow R\xrightarrow{\left[\begin{smallmatrix} 1\\ -1 \end{smallmatrix}\right]} R^2\xrightarrow{\left[\begin{smallmatrix}1&0\\0& 1\\-1&-1 \end{smallmatrix}\right]} R^3\xrightarrow{\left[\begin{smallmatrix}y&0&0\\0&1&0\\0&0&1\\-x&-x&-x \end{smallmatrix}\right]} R^4\xrightarrow{\left[\begin{smallmatrix} 1&y&0&0\\ 0&0&1&0\\0&0&0&1 \end{smallmatrix}\right]} R^3\xrightarrow{\left[\begin{smallmatrix} 1&y&0\\ 0&0&1 \end{smallmatrix}\right]} R^2\xrightarrow{\left[\begin{smallmatrix} x&y \end{smallmatrix}\right]}R\xrightarrow{}0$\end{center}with the seven nonzero modules in degrees $-6,\cdots,0$. The Koszul $4$-complex on $x,y,z$ is \begin{center}$K^\bullet(x,y,z;R):0\rightarrow R\xrightarrow{\left[\begin{smallmatrix} 1\\ -1\\1 \end{smallmatrix}\right]} R^3\xrightarrow{\left[\begin{smallmatrix}1&0&0\\0& 1&0\\0&0&1\\-1&-1&0\\1&0&-1\\0&1&1 \end{smallmatrix}\right]} R^6\xrightarrow{\left[\begin{smallmatrix} z&0&0&0&0&0\\0& 1&0&0&0&0\\0& 0&1&0&0&0 \\0& 0&0&1&0&0\\ 0& 0&0&0&1&0\\0& 0&0&0&0&1\\-y& -y&0&-y&0&0\\1& 0&-1&0&-1&0\\0& 1&1&0&0&-1\\0& 0&0&x&x&x\\\end{smallmatrix}\right]} R^{10}\xrightarrow{\left[\begin{smallmatrix} 1&z&0&0&0&0&0&0&0&0\\-1&0&z&0&0&0&0&0&0&0\\0&0&0&1&0&0&0&0&0&0 \\0&0&0&0&1&0&0&0&0&0\\ 0&0&0&0&0&1&0&0&0&0\\0&0&0&0&0&0&1&0&0&0\\0&0&0&0&0&0&0&1&0&0\\0&0&0&0&0&0&0&0&1&0\\0&0&0&0&0&0&0&0&0&1\\0&-y&-y&-y&-y&0&-1&-y&0&0\\0&1&1&0&0&-1&0&0&-1&0\\0&0&0&x&x&x&0&0&0&-1\\ \end{smallmatrix}\right]} R^{12}\xrightarrow{\left[\begin{smallmatrix} 1&0&z&0&0&0&0&0&0&0&0&0\\0&1&0&z&0&0&0&0&0&0&0&0\\-1&-1&0&0&z&0&0&0&0&0&0&0\\0&0&0&0&0&1&0&0&0&0&0&0 \\0&0&0&0&0&0&1&0&0&0&0&0\\ 0&0&0&0&0&0&0&1&0&0&0&0\\0&0&0&0&0&0&0&0&1&0&0&0\\0&0&0&0&0&0&0&0&0&1&0&0\\0&0&0&0&0&0&0&0&0&0&1&0\\0&0&0&0&0&0&0&0&0&0&0&1\\0&0&-y&-y&-y&-1&-y&-y&0&-1&-y&0\\0&0&x&x&x&0&0&0&-1&0&0&-1\\ \end{smallmatrix}\right]} R^{12}\xrightarrow{\left[\begin{smallmatrix} y&0&0&z&0&0&0&0&0&0&0&0\\0&1&0&0&z&0&0&0&0&0&0&0\\0&0&1&0&0&z&0&0&0&0&0&0\\-x&-x&-x&0&0&0&z&0&0&0&0&0 \\0&0&0&0&0&0&0&1&0&0&0&0\\ 0&0&0&0&0&0&0&0&1&0&0&0\\0&0&0&0&0&0&0&0&0&1&0&0\\0&0&0&0&0&0&0&0&0&0&1&0\\0&0&0&0&0&0&0&0&0&0&0&1\\0&0&0&-x&-xy&-xy&-y&-x&-xy&-y&-x&-y\\ \end{smallmatrix}\right]} R^{10}\xrightarrow{\left[\begin{smallmatrix} 1&y&0&0&z&0&0&0&0&0\\0&0&1&0&0&z&0&0&0&0\\0&0&0&1&0&0&z&0&0&0\\0&0&0&0&0&0&0&1&0&0\\ 0&0&0&0&0&0&0&0&1&0\\0&0& 0&0&0&0&0&0&0&1\end{smallmatrix}\right]}R^6\xrightarrow{\left[\begin{smallmatrix} 1&y&0&z&0&0\\ 0&0&1&0&z&0\\0&0&0&0&0&1 \end{smallmatrix}\right]} R^3\xrightarrow{\left[\begin{smallmatrix} x&y &z\end{smallmatrix}\right]}R\xrightarrow{}0$\end{center}with the ten nonzero modules in degrees $-9,\cdots,0$.}
\end{exa}

We next define the Koszul $N$-complex on $N$-complexes and Koszul cohomology.

\begin{df}\label{lem:1.3}{\rm Let $\emph{\textbf{x}}=x_1,\cdots,x_d$ be a sequence of elements in $R$ and $X$ an $N$-complex. The Koszul $N$-complex of $\emph{\textbf{x}}$ on $X$ is the $N$-complex
\begin{center}$K^\bullet(\emph{\textbf{x}};X):=K^\bullet(\emph{\textbf{x}};R)\otimes_RX$.\end{center}
For $t=1,\cdots,N-1$, the Koszul cohomology of $\emph{\textbf{x}}$ on $X$ is
\begin{center}$\mathrm{H}^j_t(\emph{\textbf{x}};X)=\mathrm{H}^j_t(K^\bullet(\emph{\textbf{x}};X))$ for $j\in\mathbb{Z}$.\end{center}}
\end{df}

\begin{exa}\label{lem:1.4}{\rm Let $x,y$ be two elements in $R$ and $M$ an $R$-module.

(1) For $N=3$, the Koszul $3$-complex of $x$ on $M$ is\begin{center}$K^\bullet(x;M):0\rightarrow M\stackrel{1}\rightarrow M\stackrel{x}\rightarrow M\rightarrow0$\end{center}with $M$ in degrees $-2,-1,0$.
Therefore, one has that
\begin{center}$\mathrm{H}^{-2}_1(x;M)=\mathrm{H}^{-1}_2(x;M)=0$,\end{center}
\begin{center}$\mathrm{H}^{-1}_1(x;M)=\mathrm{H}^{-2}_2(x;M)=(0:_Mx)$,\end{center}
\begin{center}$\mathrm{H}^{0}_1(x;M)=\mathrm{H}^{0}_2(x;M)=M/xM$.\end{center}The Koszul $3$-complex of $x,y$ on $M$ is\begin{center}$K^\bullet(x,y;M):0\rightarrow M\xrightarrow{\left[\begin{smallmatrix} 1\\ -1 \end{smallmatrix}\right]} M^2\xrightarrow{\left[\begin{smallmatrix}y&0\\0& 1\\-x&-x \end{smallmatrix}\right]} M^3\xrightarrow{\left[\begin{smallmatrix} 1&y&0\\ 0&0&1 \end{smallmatrix}\right]} M^2\xrightarrow{\left[\begin{smallmatrix} x&y \end{smallmatrix}\right]}M\xrightarrow{}0$.\end{center}Therefore, we conclude that \begin{center}$\mathrm{H}^{-4}_1(x,y;M)=\mathrm{H}^{-4}_2(x,y;M)=0$,\end{center}
\begin{center}$\mathrm{H}^{-3}_1(x,y;M)=\mathrm{H}^{-3}_2(x,y;M)=(0:_M(x,y))$,\end{center}
\begin{center}$\mathrm{H}^{0}_1(x,y;M)=\mathrm{H}^{0}_2(x,y;M)=M/(x,y)M$.\end{center}

(2) For $N=4$, the Koszul $4$-complex of $x$ on $M$ is\begin{center}$K^\bullet(x;M):0\rightarrow M\stackrel{1}\rightarrow M\stackrel{1}\rightarrow M\stackrel{x}\rightarrow M\rightarrow0$\end{center}with $M$ in degrees $-3,-2,-1,0$. Therefore, one has that \begin{center}$\mathrm{H}^{-3}_1(x;M)=\mathrm{H}^{-3}_2(x;M)=\mathrm{H}^{-2}_1(x;M)=\mathrm{H}^{-2}_3(x;M)
=\mathrm{H}^{-1}_2(x;M)=\mathrm{H}^{-1}_3(x;M)=0$,  $\mathrm{H}^{-3}_3(x;M)=\mathrm{H}^{-2}_2(x;M)=\mathrm{H}^{-1}_1(x;M)=(0:_Mx)$, $\mathrm{H}^{0}_1(x;M)=\mathrm{H}^{0}_2(x;M)=\mathrm{H}^{0}_3(x;M)=M/xM$.\end{center}
The Koszul $4$-complex of $x,y$ on $M$ is \begin{center}$K^\bullet(x,y;M):0\rightarrow M\xrightarrow{\left[\begin{smallmatrix} 1\\ -1 \end{smallmatrix}\right]} M^2\xrightarrow{\left[\begin{smallmatrix}1&0\\0& 1\\-1&-1 \end{smallmatrix}\right]} M^3\xrightarrow{\left[\begin{smallmatrix}y&0&0\\0&1&0\\0&0&1\\-x&-x&-x \end{smallmatrix}\right]} M^4\xrightarrow{\left[\begin{smallmatrix} 1&y&0&0\\ 0&0&1&0\\0&0&0&1 \end{smallmatrix}\right]} M^3\xrightarrow{\left[\begin{smallmatrix} 1&y&0\\ 0&0&1 \end{smallmatrix}\right]} M^2\xrightarrow{\left[\begin{smallmatrix} x&y \end{smallmatrix}\right]}M\xrightarrow{}0$.\end{center}Therefore, one obtains that\begin{center}$\mathrm{H}^{-6}_t(x,y;M)=\mathrm{H}^{-5}_t(x,y;M)=0$ for $t=1,2,3$,\end{center}
\begin{center}$\mathrm{H}^{-4}_1(x,y;M)=\mathrm{H}^{-4}_2(x,y;M)=\mathrm{H}^{-4}_3(x,y;M)=(0:_M(x,y))$, \end{center}
\begin{center}$\mathrm{H}^{0}_1(x,y;M)=\mathrm{H}^{0}_2(x,y;M)=\mathrm{H}^{0}_3(x,y;M)=M/(x,y)M$.\end{center}

(3) Let $\emph{\textbf{x}}=x_1,... ,x_d$ be a sequence
 in $R$. The Koszul $N$-complex on $x_1$ is\begin{center}$K^\bullet(x_1;M):0\rightarrow M\stackrel{1}\rightarrow M\stackrel{1}\rightarrow\cdots \stackrel{1}\rightarrow M\stackrel{x_1}\rightarrow M\rightarrow0$\end{center}with $M$ in degrees $-N+1,\cdots,0$. Therefore, one has that
\begin{center}$\mathrm{H}^{-t}_t(x_1;M)=(0:_Mx_1)$ for $t=1,\cdots,N-1$.\end{center}
\begin{center}$\mathrm{H}^{0}_t(x_1;M)=M/x_1M$ for $t=1,\cdots,N-1$.\end{center}
\begin{center}$\mathrm{H}^{-N}_t(x_1;M)=0$ for $t=1,\cdots,N-1$.\end{center}
 Consider the exact sequence of $N$-complexes
\begin{center}$0\rightarrow K^\bullet(x_1;M)\rightarrow K^\bullet(x_1,x_2;M)\rightarrow \Sigma K^\bullet(x_1;M)\rightarrow0$,\end{center} which implies that
\begin{center}$\mathrm{H}^{-N}_t(K^\bullet(x_1,x_2;M))=(0:_M(x_1,x_2))$ for $t=1,\cdots,N-1$,\end{center}
\begin{center}$\mathrm{H}^{0}_t(K^\bullet(x_1,x_2;M))=M/(x_1,x_2)M$ for $t=1,\cdots,N-1$,\end{center}
\begin{center}$\mathrm{H}^{-N-t}_t(K^\bullet(x_1,x_2;M))=0$ for $t=1,\cdots,N-1$.\end{center}
By induction, one obtains that
\begin{center}$\bigg\{\begin{array}{cc}
\mathrm{H}^{-kN-t}_t(\emph{\textbf{x}};M)=0           & d=2k \\
 \hspace{-0.32cm}\mathrm{H}^{-kN}_t(\emph{\textbf{x}};M)=0             & \hspace{0.75cm}d=2k-1 \\
\end{array}$for $t=1,\cdots,N-1$.\end{center}
\begin{center}$\bigg\{\begin{array}{cc}
 \hspace{-0.32cm}\mathrm{H}^{-kN}_t(\emph{\textbf{x}};M)=(0:_M\emph{\textbf{x}})            &  d=2k\\
 \mathrm{H}^{-kN-t}_t(\emph{\textbf{x}};M)=(0:_M\emph{\textbf{x}})             & \hspace{0.75cm}d=2k+1 \\
\end{array}$for $t=1,\cdots,N-1$,\end{center}
\begin{center}$\mathrm{H}^{0}_t(\emph{\textbf{x}};M)=M/\emph{\textbf{x}}M$ for $t=1,\cdots,N-1$,\end{center}}
\end{exa}

\begin{prop}\label{lem:1.8}{\it{Given a sequence of elements $\textbf{x}=(x_1,\cdots,x_d)$ in $R$, one has an isomorphism in $\mathbf{K}_N(R)$
\begin{center}$K^\bullet(\textbf{x};R)\cong\Sigma^{d}\mathrm{Hom}_R(K^\bullet(\textbf{x};R),R)$.\end{center}}}
\end{prop}
\begin{proof} For $x_1$ there exists an exact triangle $R\xrightarrow{x_1}R\xrightarrow{} K^\bullet(x_1;R)\rightarrow\Sigma R$ in $\mathbf{K}_N(R)$. Applying the functor $\mathrm{RHom}_R(-,R)$ to this triangle, one gets an exact triangle\begin{center}$\Sigma^{-1}R\rightarrow \mathrm{Hom}_R(K^\bullet(x_1;R),R)\rightarrow R\xrightarrow{x_1}R$.\end{center}Thus
$K^\bullet(x_1;R)\cong\Sigma\mathrm{Hom}_R(K^\bullet(x_1;R),R)$ in $\mathbf{K}_N(R)$. For $x_2$ there exists an exact triangle $K^\bullet(x_1;R)\xrightarrow{x_2}K^\bullet(x_1;R)\xrightarrow{} K^\bullet(x_1,x_2;R)\rightarrow\Sigma K^\bullet(x_1;R)$ in $\mathbf{K}_N(R)$. Applying the functor $\mathrm{RHom}_R(-,R)$ to this triangle, one gets an exact triangle\begin{center}$\Sigma^{-1}\mathrm{Hom}_R(K^\bullet(x_1;R),R)\rightarrow
\mathrm{Hom}_R(K^\bullet(x_1,x_2;R),R)
\rightarrow\mathrm{Hom}_R(K^\bullet(x_1;R),R)\rightarrow\mathrm{Hom}_R(K^\bullet(x_1;R),R)$,\end{center}
 which implies that
$K^\bullet(x_1,x_2;R)\cong\Sigma^{2}\mathrm{Hom}_R(K^\bullet(x_1,x_2;R),R)$ in $\mathbf{K}_N(R)$. Continuing this process, we obtain the isomorphism we seek.
\end{proof}

\section{\bf The $\check{\mathrm{C}}$ech $N$-complex}
For $x\in R$, the localization $R_x$
is obtained by inverting the multiplicatively closed set $\{1,x,x^2,\cdots\}$. Let $\iota:R\rightarrow R_x$ be the canonical map sending each $r\in R$ to the class of the fraction
$r/1\in R_x$. This section gives a construction of $\check{\mathrm{C}}$ech $N$-complexes.

\begin{con}\label{lem:2.1}{\rm Let $x$ be an element in $R$. We have a commutative diagram:
\begin{center} $\xymatrix@C=23pt@R=18pt{
\vdots\ar[d]&\vdots \ar[d]^x& \vdots \ar[d]^x &&\vdots\ar[d]^x&\vdots\ar@{=}[d]\\
K^\bullet(x^3;R):\ar[d]\ 0\ar[r] & R\ar@{=}[r] \ar[d]^x& R\ar@{=}[r] \ar[d]^x&\cdots\ar@{=}[r]& R\ar[r]^{x^3} \ar[d]^x& R\ar[r] \ar@{=}[d]& 0\\
K^\bullet(x^2;R):\ar[d]\ 0\ar[r] & R\ar@{=}[r] \ar[d]^x& R\ar@{=}[r] \ar[d]^x&\cdots\ar@{=}[r]& R\ar[r]^{x^2} \ar[d]^x& R\ar[r] \ar@{=}[d]& 0\\
K^\bullet(x;R):\ 0\ar[r] & R\ar@{=}[r] & R\ar@{=}[r] &\cdots\ar@{=}[r]& R\ar[r]^x & R\ar[r] & 0\\}$
\end{center}
Applying the functor $\mathrm{Hom}_R(-,R)$ to this diagram, we have a commutative diagram:
\begin{center} $\xymatrix@C=23pt@R=18pt{
 0\ar[r] & R\ar[r]^x \ar@{=}[d]& R\ar@{=}[r] \ar[d]^x&\cdots\ar@{=}[r]& R\ar@{=}[r] \ar[d]^x& R\ar[r] \ar[d]^x& 0\\
  0\ar[r] & R\ar[r]^{x^2} \ar@{=}[d]& R\ar@{=}[r] \ar[d]^x&\cdots\ar@{=}[r]& R\ar@{=}[r] \ar[d]^x& R\ar[r] \ar[d]^x& 0\\
 0\ar[r] & R\ar[r]^{x^3}\ar@{=}[d] & R\ar@{=}[r]\ar[d]^x &\cdots\ar@{=}[r]& R\ar@{=}[r]\ar[d]^x & R\ar[r]\ar[d]^x & 0\\
&\vdots & \vdots  & &\vdots&\vdots\\}$
\end{center}
In the limit we get the following $N$-complex
\begin{center}$0\rightarrow R\xrightarrow{\iota}R_x\xrightarrow{1}R_x\xrightarrow{1}\cdots\xrightarrow{1}R_x\xrightarrow{}0$\end{center} with modules $R$ in degree $0$ and $R_x$ in degrees $1,\cdots,N-1$, which is called
the $\check{\mathrm{C}}$ech $N$-complex on $x$, denoted by $\check{C}^\bullet(x;R)$. We also have an exact triangle in $\mathbf{K}_N(R)$
\begin{center}$\Sigma^{-1}R_x\rightarrow \check{C}^\bullet(x;R)\rightarrow R\xrightarrow{\iota}R_x$.\end{center}Let $x,y$ be two elements in $R$. Then the morphisms
\begin{center} $\xymatrix@C=23pt@R=18pt{
K^\bullet(x^s;R)\ar[d]^{u}:\ 0\ar[r] & R\ar@{=}[r] \ar[d]^{xy}& R\ar@{=}[r] \ar[d]^{xy}&\cdots\ar@{=}[r]& R\ar[r]^{x^s} \ar[d]^{xy}& R\ar[r] \ar[d]^y& 0\\
K^\bullet(x^{s-1};R):\ 0\ar[r] & R\ar@{=}[r] & R\ar@{=}[r] &\cdots\ar@{=}[r]& R\ar[r]^{x^{s-1}} & R\ar[r] & 0\\}$
\end{center}
\begin{center} $\xymatrix@C=23pt@R=20pt{
K^\bullet(x^s;R)\ar[d]^{v}:\ 0\ar[r] & R\ar@{=}[r] \ar[d]^{x}& R\ar@{=}[r] \ar[d]^{x}&\cdots\ar@{=}[r]& R\ar[r]^{x^s} \ar[d]^{x}& R\ar[r] \ar@{=}[d]& 0\\
K^\bullet(x^{s-1};R):\ 0\ar[r] & R\ar@{=}[r] & R\ar@{=}[r] &\cdots\ar@{=}[r]& R\ar[r]^{x^{s-1}} & R\ar[r] & 0\\}$
\end{center}induce a commutative diagram in $\mathbf{C}_N(R)$:
\begin{center} $\xymatrix@C=23pt@R=18pt{
0\ar[r] & K^\bullet(x^s;R)\ar[r] \ar[d]^{v}& K^\bullet(x^s,y^s;R)\ar[r]\ar@{.>}[d]& \Sigma K^\bullet(x^s;R)\ar[d]^{\Sigma u}\ar[r]&0\\
0\ar[r] & K^\bullet(x^{s-1};R)\ar[r]& K^\bullet(x^{s-1},y^{s-1};R)\ar[r] & \Sigma K^\bullet(x^{s-1};R)\ar[r]&0}$\end{center}Applying the functor $\mathrm{Hom}_R(-,R)$ to the diagram, one obtains a direct system
\begin{center}$\mathrm{Hom}_R(K^\bullet(x,y;R),R)\rightarrow\mathrm{Hom}_R(K^\bullet(x^2,y^2;R),R)\rightarrow \mathrm{Hom}_R(K^\bullet(x^3,y^3;R),R)\rightarrow\cdots$.\end{center}In the limit we get an $N$-complex $\underrightarrow{\textrm{lim}}\mathrm{Hom}_R(K^\bullet(x^s,y^s;R),R)$, which is called
the $\check{\mathrm{C}}$ech $N$-complex on $x,y$, denoted by $\check{C}^\bullet(x,y;R)$. We also have an exact triangle in $\mathbf{K}_N(R)$
\begin{center}$\Sigma^{-1}\check{C}^\bullet(x;R)_y\rightarrow \check{C}^\bullet(x,y;R)\rightarrow \check{C}^\bullet(x;R)\xrightarrow{\iota}\check{C}^\bullet(x;R)_y$.\end{center}
For a sequence $\emph{\textbf{x}}=(x_1,\cdots,x_d)$ of elements in $R$, set  $\emph{\textbf{x}}^s=x^s_1,\cdots,x^s_d$ and $\emph{\textbf{y}}=(x_1,\cdots,x_{d-1})$.
By induction, the $\check{\mathrm{C}}$ech $N$-complex $\check{C}^\bullet(\emph{\textbf{x}};R)$ on $\emph{\textbf{x}}$ is $\underrightarrow{\textrm{lim}}\mathrm{Hom}_R(K^\bullet(\emph{\textbf{x}}^s;R),R)$
 and we have an exact triangle in $\mathbf{K}_N(R)$
\begin{center}$\Sigma^{-1}\check{C}^\bullet(\emph{\textbf{y}};R)_{x_d}\rightarrow \check{C}^\bullet(\emph{\textbf{x}};R)\rightarrow \check{C}^\bullet(\emph{\textbf{y}};R)\xrightarrow{\iota}\check{C}^\bullet(\emph{\textbf{y}};R)_{x_d}$.\end{center}
In fact, by induction, one can obtain the following isomorphism \begin{center}$\check{C}^\bullet(\emph{\textbf{x}};R)
\cong\check{C}^\bullet(x_1;R)\otimes_R\cdots\otimes_R\check{C}^\bullet(x_d;R)$.\end{center}}
\end{con}

\begin{exa}\label{lem:2.2}{\rm Let $x,y$ be two elements in $R$.

For $N=3$, the $\check{\mathrm{C}}$ech $3$-complex on $x$ is\begin{center}$\check{C}^\bullet(x;R):0\rightarrow R\stackrel{\iota_x}\rightarrow R_x\stackrel{1}\rightarrow R_x\rightarrow0$.\end{center}Therefore, the $\check{\mathrm{C}}$ech $3$-complex on $x,y$ is \begin{center}$\check{C}^\bullet(x,y;R):0\rightarrow R\xrightarrow{\left[\begin{smallmatrix} \iota_x\\ \iota_y \end{smallmatrix}\right]} R_x\oplus R_y\xrightarrow{\left[\begin{smallmatrix}1&0\\\iota_y& 0\\0&1 \end{smallmatrix}\right]} R_x\oplus R_{xy}\oplus R_y\xrightarrow{\left[\begin{smallmatrix} \iota_y&0&-\iota_x\\ 0&1&-\iota_x\end{smallmatrix}\right]} R_{xy}\oplus R_{xy}\xrightarrow{\left[\begin{smallmatrix} 1&-1 \end{smallmatrix}\right]}R_{xy}\xrightarrow{}0$,\end{center}where the five nonzero modules are in degrees $0,1,2,3,4$.

For $N=4$, the $\check{\mathrm{C}}$ech $4$-complex on $x$ is\begin{center}$\check{C}^\bullet(x;R):0\rightarrow R\stackrel{\iota_x}\rightarrow R_x\stackrel{1}\rightarrow R_x\stackrel{1}\rightarrow  R_x\rightarrow0$.\end{center}Therefore, the $\check{\mathrm{C}}$ech $4$-complex on $x,y$ is \begin{center}$\check{C}^\bullet(x,y;R):0\rightarrow R\xrightarrow{\left[\begin{smallmatrix} \iota_x\\ \iota_y \end{smallmatrix}\right]} R_x\oplus R_y\xrightarrow{\left[\begin{smallmatrix}1&0\\ \iota_y& 0\\0&1 \end{smallmatrix}\right]} R_x\oplus R_{xy}\oplus R_y\xrightarrow{\left[\begin{smallmatrix}1&0&0\\ \iota_y&0&0\\0&1&0\\0&0&1 \end{smallmatrix}\right]} R_x\oplus R_{xy}\oplus R_{xy}\oplus R_y\xrightarrow{\left[\begin{smallmatrix}\iota_y&0&0&-\iota_x\\0&1&0&-\iota_x\\0&0&1&-\iota_x \end{smallmatrix}\right]} R_{xy}\oplus R_{xy}\oplus R_{xy}\xrightarrow{\left[\begin{smallmatrix}1&0&-1\\0& 1&-1 \end{smallmatrix}\right]} R_{xy}\oplus R_{xy}\xrightarrow{\left[\begin{smallmatrix} 1&-1 \end{smallmatrix}\right]}R_{xy}\xrightarrow{}0$,\end{center}where the seven nonzero modules are in degrees $0,1,2,3,4,5,6$.}
\end{exa}

\begin{lem}\label{lem:3.7}{\it{For a sequence $\textbf{x}=x_1,\cdots,x_d$ in $R$, the natural morphism $e:\check{C}(\textbf{x};R)\rightarrow R$ $(R$ is viewed as the $N$-complex $D^0_1(R))$ induces a quasi-isomorphism \begin{center}$e\otimes 1,1\otimes e:\check{C}(\textbf{x};R)\otimes_R\check{C}(\textbf{x};R)\stackrel{\simeq}\rightarrow \check{C}(\textbf{x};R)$.\end{center}}}
\end{lem}
\begin{proof} By symmetry it is enough to look only at
\begin{center}$1\otimes e:\check{C}(\emph{\textbf{x}};R)\otimes_R\check{C}(\emph{\textbf{x}};R)\rightarrow \check{C}(\emph{\textbf{x}};R)$.\end{center}Since the $N$-complexes $\check{C}(x_i;R)$ are semi-flat, it is enough to consider the case $d=1$ and $x=x_1$. We have the following commutative diagram:
\begin{center} $\xymatrix@C=20pt@R=18pt{
&0\ar[d] & 0 \ar[d] &0\ar[d]& \\
0\ar[r]&\Sigma^{-2}R_x\ar[d]\ar[r]  & \Sigma^{-1}\check{C}(x;R)_x \ar[d]\ar[r] &\Sigma^{-1}R_x\ar[d]\ar[r]&0 \\
0\ar[r]&\Sigma^{-1}\check{C}(x;R)_x\ar[d]\ar[r]  & \check{C}(x;R)\otimes_R\check{C}(x;R) \ar[d]\ar[r]^{\ \ \ \ \ \ \ 1\otimes e} & \check{C}(x;R)\ar[d]\ar[r]&0\\
0\ar[r]&\Sigma^{-1}R_x\ar[r]\ar[d]& \check{C}(x;R)\ar[r]\ar[d] & R\ar[d]\ar[r]&0 \\
&0&0 & 0 }$
\end{center}
Note that $x:R_x\rightarrow R_x$ is an isomorphism in $\mathbf{D}_N(R)$, it follows that $\Sigma^{-1}\check{C}(x;R)_x$ is acyclic. This completes the proof.
\end{proof}

Given an $N$-complex $X$, set $\check{C}^\bullet(\emph{\textbf{x}};X):=\check{C}^\bullet(\emph{\textbf{x}};R)\otimes_RX$. The $R$-module
\begin{center}$\check{\mathrm{H}}^j_t(\emph{\textbf{x}};X)=\mathrm{H}^j_t(\check{C}(\emph{\textbf{x}};X))$ for $t=1,\cdots, N-1$\end{center}
is the $j$-th $\check{\mathrm{C}}$ech cohomology of $\emph{\textbf{x}}$ on $X$.

\begin{exa}\label{lem:3.9}{\rm (1) Let $x$ be an element in $R$. Then
\begin{center}$\begin{aligned}\mathrm{\check{H}}^{0}_t(x;R)
&=\{r\in R\hspace{0.03cm}|\hspace{0.03cm}r/1=0\ \textrm{in}\ R_x\}\\
&=\{r\in R\hspace{0.03cm}|\hspace{0.03cm}x^sr=0\ \textrm{for\ some}\ s\geq0\}\\
&=\bigcup_{s\geq0}(0:_Rx^s)\cong\underrightarrow{\textrm{lim}}\mathrm{Hom}_R(R/(x^s),R),\end{aligned}$\end{center}
 \begin{center}$\mathrm{\check{H}}^{t}_{N-t}(x;R)=R_x/R$ for $t=1,\cdots,N-1$.\end{center}

(2) Let $\emph{\textbf{x}}=x_1,\cdots,x_d$ be a sequence of element
in $R$ and $M$ an $R$-module. The $\check{\mathrm{C}}$ech $N$-complex of $x_1$ on $M$ is\begin{center}$\check{C}(x_1;M):0\rightarrow M\xrightarrow{\iota}M_{x_1}\xrightarrow{1}M_{x_1}\xrightarrow{1}\cdots\xrightarrow{1}M_{x_1}\xrightarrow{}0$\end{center} with modules $M$ in degree $0$ and $M_{x_1}$ in degrees $1,\cdots,N-1$, Therefore, one has that
\begin{center}$\mathrm{\check{H}}^{0}_t(x_1;M)\cong\underrightarrow{\textrm{lim}}\mathrm{Hom}_R(R/(x^s),M)$ for $t=1,\cdots,N-1$.\end{center}
\begin{center}$\mathrm{\check{H}}^{t}_{N-t}(x_1;M)=M_{x_1}/M$ for $t=1,\cdots,N-1$.\end{center}
\begin{center}$\mathrm{\check{H}}^{N}_t(x_1;M)=0$ for $t=1,\cdots,N-1$.\end{center}
 Consider the exact sequence of $N$-complexes
\begin{center}$0\rightarrow \Sigma^{-1}\check{C}(x_1;M)_{x_2}\rightarrow \check{C}(x_1,x_2;M)\rightarrow \check{C}(x_1;M)\rightarrow0$,\end{center} which implies that
\begin{center}$\mathrm{\check{H}}^{0}_t(x_1,x_2;M)\cong
\underrightarrow{\textrm{lim}}\mathrm{Hom}_R(R/(x_1^s,x_2^s),M)$ for $t=1,\cdots,N-1$.\end{center}
\begin{center}$\mathrm{\check{H}}^{N}_t(x_1;M)=M_{x_1,x_2}/(\mathrm{Im}M_{x_1}+\mathrm{Im}M_{x_2})$ for $t=1,\cdots,N-1$.\end{center}
\begin{center}$\mathrm{\check{H}}^{N+t}_{N-t}(x_1,x_2;M)=0$ for $t=1,\cdots,N-1$.\end{center}
By induction, one obtains that
\begin{center}$\bigg\{\begin{array}{cc}
 \hspace{-1.15cm}\mathrm{\check{H}}^{j}_t(\emph{\textbf{x}};M)=0\ \textrm{for}\ j\geq kN         & \hspace{0.75cm}d=2k-1 \\
\mathrm{\check{H}}^{j}_{N-t}(\emph{\textbf{x}};M)=0\ \textrm{for}\ j\geq kN+t         & d=2k \\
\end{array}$for $1\leq t\leq N-1$,\end{center}
\begin{center}$\bigg\{\begin{array}{cc}
\mathrm{\check{H}}^{(k-1)N+t}_{N-t}(\emph{\textbf{x}};M)=M_{x_1\cdots x_d}/\Sigma^c_{i=1}\mathrm{image}M_{x_1\cdots x_{i-1}x_{i+1}\cdots x_d}            &  d=2k-1\\
 \hspace{-1.0cm}\mathrm{\check{H}}^{kN}_t(\emph{\textbf{x}};M)=M_{x_1\cdots x_d}/\Sigma^c_{i=1}\mathrm{image}M_{x_1\cdots x_{i-1}x_{i+1}\cdots x_d}           & \hspace{-0.66cm}d=2k \\
\end{array}$for $1\leq t\leq N-1$,\end{center}
\begin{center}$\mathrm{\check{H}}^{0}_t(\emph{\textbf{x}};M)\cong
\underrightarrow{\textrm{lim}}\mathrm{Hom}_R(R/(\emph{\textbf{x}}^s),M)$ for $1\leq t\leq N-1$.\end{center}}
\end{exa}

\section{\bf Derived torsion of $N$-complexes}
In this section, we explicit derived torsion functors in $\mathbf{D}_N(R)$ using the $\check{\mathrm{C}}$ech $N$-complex.

Let $\mathfrak{a}$ be an ideal of $R$.
For each $R$-module $M$, set
\begin{center}$\Gamma_\mathfrak{a}(M)=\{m\in M\hspace{0.03cm}|\hspace{0.03cm}\mathfrak{a}^nm=0\ \textrm{for\ some\ integer}\ n\}$.\end{center}
 There is a functorial homomorphism $\sigma_M:\Gamma_\mathfrak{a}(M)\rightarrow M$ which is just the
inclusion. When they coincide, $M$ is said to
be $\mathfrak{a}$-torsion. The association $M\rightarrow \Gamma_\mathfrak{a}(M)$ extends to define a left exact additive functor on $\mathbf{C}_N(R)$, it is called
the $\mathfrak{a}$-torsion functor. By \cite[Corollary 3.29]{IKM}, the functor $\Gamma_\mathfrak{a}$ has a right derived functor \begin{center}$\mathrm{R}\Gamma_\mathfrak{a}:\mathbf{D}_N(R)\rightarrow\mathbf{D}_N(R),\ \xi:\Gamma_\mathfrak{a}\rightarrow\mathrm{R}\Gamma_\mathfrak{a}$\end{center}constructed using semi-injective resolutions.

\begin{prop}\label{lem:2.0}{\it{There is a functorial morphism $\sigma^\mathrm{R}_X:\mathrm{R}\Gamma_\mathfrak{a}(X)\rightarrow X$, such that $\sigma_X=\sigma^\mathrm{R}_X\circ\xi_X$ as morphisms $\Gamma_\mathfrak{a}(X)\rightarrow X$ in $\mathbf{D}_N(R)$.}}
\end{prop}
\begin{proof} Let $X\stackrel{\alpha}\rightarrow I$ be a semi-injective resolution, and define $\sigma^\mathrm{R}_X=\alpha^{-1}\circ\sigma_I\circ\xi^{-1}_I\circ\mathrm{R}\Gamma_\mathfrak{a}(\alpha)$. This is independent of the resolution.
\end{proof}

For each $N$-complex $X$ and $i\in\mathbb{Z}$, the
$i$th local cohomology of $X$ with support in $\mathfrak{a}$ is \begin{center}$\mathrm{H}^i_{t,\mathfrak{a}}(X)=\mathrm{H}^i_t(\mathrm{R}\Gamma_\mathfrak{a}(X))$ for $t=1,\cdots,N-1$.\end{center}

\begin{exa}\label{lem:3.2}{\rm Let $R=\mathbb{Z}$ and $p$ be a prime number, and let $M$ be an  indecomposable $R$-module.
By the fundamental theorem of Abelian groups, $M$
is isomorphic to $\mathbb{Z}/d\mathbb{Z}$ where either $d=0$ or $d$ is a prime power. In either
case the 3-complex \begin{center}$0\rightarrow \mathbb{Q}/d\mathbb{Z}\rightarrow \mathbb{Q}/\mathbb{Z}\stackrel{1}\rightarrow \mathbb{Q}/\mathbb{Z}\rightarrow0$\end{center}
is an injective resolution of $\mathbb{Z}/d\mathbb{Z}$. In what follows, $\mathbb{Z}_p$ denotes $\mathbb{Z}$ with $p$ inverted.

\textbf{Case 1.} If $M=\mathbb{Z}/p^e\mathbb{Z}$ for some integer $e\geq1$, then applying $\Gamma_{(p)}(-)$
to the resolution above yields the 3-complex
\begin{center}$0\rightarrow\mathbb{Z}_p/p^e\mathbb{Z}\rightarrow\mathbb{Z}_p/\mathbb{Z}\stackrel{1}\rightarrow\mathbb{Z}_p/\mathbb{Z}\rightarrow0$.\end{center}
 Hence one obtains that \begin{center}$\mathrm{H}^0_{1,(p)}(M)=\mathrm{H}^0_{2,(p)}(M)=\mathbb{Z}/p^e\mathbb{Z}=M$,\end{center}
 \begin{center}$\mathrm{H}^1_{1,(p)}(M)=\mathrm{H}^1_{2,(p)}(M)=0
 =\mathrm{H}^2_{1,(p)}(M)=\mathrm{H}^2_{2,(p)}(M)$.\end{center}

 \textbf{Case 2.} If $M=\mathbb{Z}/d\mathbb{Z}$ with $d$ nonzero and relatively prime to $p$, then applying $\Gamma_{(p)}(-)$
to the resolution above yields the 3-complex
\begin{center}$0\rightarrow d\mathbb{Z}_p/d\mathbb{Z}\rightarrow\mathbb{Z}_p/\mathbb{Z}\stackrel{1}\rightarrow\mathbb{Z}_p/\mathbb{Z}\rightarrow0$.\end{center}
Thus we conclude that \begin{center}$\mathrm{H}^1_{2,(p)}(M)=\mathbb{Z}_p/(d\mathbb{Z}_p+\mathbb{Z})=\mathrm{H}^2_{1,(p)}(M)$, \end{center}
 \begin{center} $\mathrm{H}^0_{1,(p)}(M)=\mathrm{H}^0_{2,(p)}(M)=0
 =\mathrm{H}^1_{1,(p)}(M)=\mathrm{H}^2_{2,(p)}(M)$.\end{center}

 \textbf{Case 3.} If $M=\mathbb{Z}$, then applying $\Gamma_{(p)}(-)$
to the resolution above yields the 3-complex
\begin{center}$0\rightarrow 0\rightarrow\mathbb{Z}_p/\mathbb{Z}\stackrel{1}\rightarrow\mathbb{Z}_p/\mathbb{Z}\rightarrow0$.\end{center}
 Therefore, one has that \begin{center}$\mathrm{H}^1_{2,(p)}(M)=\mathbb{Z}_p/\mathbb{Z}=\mathrm{H}^2_{1,(p)}(M)$,\end{center}
 \begin{center} $\mathrm{H}^0_{1,(p)}(M)=\mathrm{H}^0_{2,(p)}(M)=0=\mathrm{H}^1_{1,(p)}(M)=\mathrm{H}^2_{2,(p)}(M)$.\end{center}}
\end{exa}

Following \cite{S}, an inverse system $\{M_i\}_{i\in\mathbb{N}}$ of abelian groups, with transition maps $p_{j,i}: M_j\rightarrow M_i$,
is called pro-zero if for every $i$ there exists $j\geq i$ such that $p_{j,i}$ is zero.

\begin{df}\label{lem:3.1} {\rm (1) Let $\emph{\textbf{x}}=x_1,\cdots,x_d$ be a sequence of elements in $R$. The sequence $\emph{\textbf{x}}$ is called a weakly
proregular sequence if for every $i<0$ and $t=1,\cdots,N-1$ the inverse system $\{\mathrm{H}^i_t(K^\bullet(\emph{\textbf{x}}^s;R))\}_{s\in\mathbb{N}}$ is pro-zero.

(2) An ideal $\mathfrak{a}$ of $R$ is called a weakly proregular ideal if it is generated by some weakly proregular sequence.}
\end{df}

Let $x$ be an element in $R$. For any $s>0$, we have a morphism of $N$-complexes
\begin{center} $\xymatrix@C=20pt@R=18pt{
K^\bullet(x^s;R)\ar[d]^{\gamma^s}:0\ar[r]  &R\ar[d]\ar@{=}[r]&R\ar[d]\ar@{=}[r] &\cdots\ar@{=}[r] &R\ar[r]^{x^s}\ar[d]&R\ar[r]\ar[d]^{\pi^s}&0 \\
D^0_1(R/(x^s)):0\ar[r]  &0\ar[r] &0\ar[r] &\cdots\ar[r] &0\ar[r]& R/(x^s)\ar[r]&0}$
\end{center}
which induces the following morphism of inverse systems\begin{equation}
\begin{split} \label{diag02}\xymatrix@C=20pt@R=17pt{
\cdots \ar[r]&K^\bullet(x^3;R)\ar[d]^{\gamma^3}\ar[r]  & K^\bullet(x^2;R) \ar[d]^{\gamma^2}\ar[r]& K^\bullet(x;R)\ar[d]^{\gamma^1} \\
\cdots \ar[r]&R/(x^3)\ar[r]& R/(x^2)\ar[r] & R/(x)}\end{split}
\end{equation}where $R/(x^s)$ is viewed as the $N$-complex $D^0_1(R/(x^s))$.
Let $X$ be an $N$-complex. (4.1) yields a morphism of direct systems:
\begin{center} $\xymatrix@C=15pt@R=18pt{
\mathrm{Hom}_R(R/(x),X) \ar[d]\ar[r]&\mathrm{Hom}_R(R/(x^2),X)\ar[d]\ar[r] & \mathrm{Hom}_R(R/(x^3),X)\ar[d]\ar[r] & \cdots\\
\mathrm{Hom}_R(K^\bullet(x;R),X)\ar[r]&\mathrm{Hom}_R(K^\bullet(x^2;R),X)\ar[r]& \mathrm{Hom}_R(K^\bullet(x^3;R),X)\ar[r] & \cdots}$
\end{center}
This gives rise to a functorial morphism of $N$-complexes
\begin{center}$\delta_{x,X}:\underrightarrow{\textrm{lim}}_{s>0}\mathrm{Hom}_R(R/(x^s),X)
\rightarrow\underrightarrow{\textrm{lim}}_{s>0}\mathrm{Hom}_R(K^\bullet(x^s;R),X)
\cong\check{C}(x;R)\otimes_RX$.\end{center}

Let $\emph{\textbf{x}}=x_1,\cdots,x_d$ be elements in $R$. The Koszul
$N$-complex on $\emph{\textbf{x}}^s$
is the $N$-complex $K^\bullet(\emph{\textbf{x}}^s;R)$. This is equipped with a morphism of $N$-complexes of $\theta^s:K^\bullet(\emph{\textbf{x}}^s;R)\rightarrow R/(\emph{\textbf{x}}^s)$. Therefore,
we obtain an inverse system of
$N$-complexes\begin{center}$\cdots\rightarrow K^\bullet(\emph{\textbf{x}}^{s+1};R)\rightarrow K^\bullet(\emph{\textbf{x}}^s;R)\rightarrow\cdots\rightarrow K^\bullet(\emph{\textbf{x}};R)$,\end{center}compatible with the morphisms $\theta^s$ and natural maps $R/(\emph{\textbf{x}}^{s+1})\rightarrow R/(\emph{\textbf{x}}^s)$.

The next result provide an explicit formula for computing $\mathrm{R}\Gamma_\mathfrak{a}$.

\begin{thm}\label{lem:3.6}{\it{Let $\textbf{x}=x_1,\cdots,x_d$ be a weakly proregular sequence in $R$ and $\mathfrak{a}$ the ideal generated by $\textbf{x}$. For any $N$-complex $X$, there is a functorial quasi-isomorphism \begin{center}$\delta^\mathrm{R}_{\textbf{x},X}:\mathrm{R}\Gamma_{\mathfrak{a}}(X)\rightarrow \check{C}(\textbf{x};R)\otimes_RX$.\end{center}}}
\end{thm}
\begin{proof} If $d=1$,
 then by construction of $\gamma^s$, we have a quasi-isomorphism of $N$-complexes
\begin{center} $\xymatrix@C=23pt@R=20pt{
 H_s:\ 0\ar[r]  &(0:_Rx^s_1)\ar@{^{(}->}[d]\ar@{=}[r]&(0:_Rx^s_1)\ar@{^{(}->}[d]\ar@{=}[r] &\cdots\ar@{=}[r] &(0:_Rx^s_1)\ar[r]\ar@{^{(}->}[d]&0\ar[r]\ar[d]&0 \\
\mathrm{Ker}\gamma^s:\ 0\ar[r]  &R\ar@{=}[r]&R\ar@{=}[r] &\cdots\ar@{=}[r] &R\ar[r]^{x^s_1}&(x^s_1)\ar[r]&0 }$
\end{center}Write $F(Y):=\Gamma_{(x_1)}(Y)$ and $G(Y):=\underrightarrow{\textrm{lim}}_{s>0}\mathrm{Hom}_R(K^\bullet(x^s_1;R),Y)$ for any $N$-complex $Y$. Let $I$ be a semi-injective $N$-complex. It is enough to
 show that $\delta_{x_1,I}:F(I)\rightarrow G(I)$ is a quasi-isomorphism. By Lemma \ref{lem:3.3} we
may assume that $I$ is a single injective module.
 For each $j\in\mathbb{Z}$ and $t=1,\cdots,N-1$, one has that \begin{center}$\begin{aligned}\mathrm{H}^j_t(\underrightarrow{\textrm{lim}}_{s>0}\mathrm{Hom}_R(\mathrm{Ker}\gamma^s,I))
 &\cong\underrightarrow{\textrm{lim}}_{s>0}\mathrm{H}^j_t(\mathrm{Hom}_R(H_s,I))\\
 &\cong\underrightarrow{\textrm{lim}}_{s>0}\mathrm{Hom}_R(\mathrm{H}^{-j}_{N-t}(H_s),I)\\
&\cong\underrightarrow{\textrm{lim}}_{s>0}\mathrm{Hom}_R(\mathrm{H}^{-j}_{N-t}(K^\bullet(x_1;R)),I).
\end{aligned}$\end{center} Since $x_1$ is weakly proresular, it follows that $\underrightarrow{\textrm{lim}}_{s>0}\mathrm{Hom}_R(\mathrm{H}^{-j}_{N-t}(K^\bullet(x_1;R)),I)=0$
for $j>0$ and $t=1,\cdots,N-1$. Hence $\underrightarrow{\textrm{lim}}_{s>0}\mathrm{Hom}_R(\mathrm{Ker}\gamma^s,I)$ is acyclic, and so $F(I)\cong G(I)$ in $\mathbf{D}_N(R)$. Now assume $d>1$ and proceed by induction on $d$. Let $X\xrightarrow{\simeq}I$ be a semi-injective resolution and
set $\emph{\textbf{y}}=x_1,\cdots,x_{d-1}$. One has the following isomorphisms
 \begin{center}$\begin{aligned}\mathrm{R}\Gamma_{\mathfrak{a}}(X)
&\cong\underrightarrow{\textrm{lim}}_{s>0}\mathrm{Hom}_R(R/(\emph{\textbf{x}}^s),I)\\
&\cong\underrightarrow{\textrm{lim}}_{s>0}\mathrm{Hom}_R(R/(\emph{\textbf{y}}^s)\otimes_RR/(x_d^s),I)\\
&\cong\underrightarrow{\textrm{lim}}_{s>0}\mathrm{Hom}_R(R/(\emph{\textbf{y}}^s),\underrightarrow{\textrm{lim}}_{s>0}\mathrm{Hom}_R(R/(x_d^s),I))\\
&\cong\underrightarrow{\textrm{lim}}_{s>0}\mathrm{Hom}_R(K^\bullet(\emph{\textbf{y}}^s;R),\underrightarrow{\textrm{lim}}_{s>0}\mathrm{Hom}_R(K^\bullet(x_d^s;R),I))\\
&\cong\underrightarrow{\textrm{lim}}_{s>0}\mathrm{Hom}_R(K^\bullet(\emph{\textbf{y}}^s;R)\otimes_RK^\bullet(x_d^s;R),I)\\
&\cong\underrightarrow{\textrm{lim}}_{s>0}\mathrm{Hom}_R(K^\bullet(\emph{\textbf{x}}^s;R),I)\\
&\cong\check{C}(\emph{\textbf{x}};R)\otimes_RX,\end{aligned}$\end{center}
 where the second one holds as $R/(\emph{\textbf{x}}^s)\cong R/(\emph{\textbf{y}}^s)\otimes_RR/(x_d^s)$, the third one is Hom-tensor adjointness, the fourth one is by induction, as claimed.
\end{proof}

\begin{cor}\label{lem:3.8}{\it{Let $\textbf{x}=x_1,\cdots,x_d$ be a set of generators for a weakly proregular ideal $\mathfrak{a}$, and let $X$ be an $N$-complex.

$\mathrm{(1)}$ The morphism
$\sigma^\mathrm{R}_{\mathrm{R}\Gamma_\mathfrak{a}(X)}:
\mathrm{R}\Gamma_\mathfrak{a}(\mathrm{R}\Gamma_\mathfrak{a}(X))\rightarrow\mathrm{R}\Gamma_\mathfrak{a}(X)$
is an isomorphism. Thus the functor
$\mathrm{R}\Gamma_\mathfrak{a}:\mathbf{D}_N(R)\rightarrow \mathbf{D}_N(R)$
is idempotent.

$\mathrm{(2)}$ There is a natural isomorphism $\mathrm{H}^i_{t,\mathfrak{a}}(X)\cong\mathrm{\check{H}}^i_t(\textbf{x};X)$.}}
\end{cor}

\section{\bf The Telescope $N$-complexes and derived completion}
This section introduces the notion of Telescope $N$-complexes and explicits the derived completion functor in $\mathbf{D}_N(R)$ using it.

\begin{df}\label{lem:4.0}{\rm Let $\{e_i\hspace{0.03cm}|\hspace{0.03cm}i\geq 0\}$ be the basis of the countably generated free $R$-module $\bigoplus_{i=0}^\infty R$. Given an element $x\in R$, define the morphism $v:D^0_1(\bigoplus_{i=0}^\infty R)\rightarrow D^0_1(\bigoplus_{i=0}^\infty R)$ of $N$-complexes by
\begin{center}$v(e_i)=\bigg\{\begin{array}{cc}
  \hspace{-1.45cm}e_0             & \textrm{if}\ i=0 \\
  e_{i-1}-xe_i             & \hspace{0.1cm}\textrm{if}\ i\geq1.\\
\end{array}$\end{center} The telescope $N$-complex $\mathrm{Tel}(x;R)$ is the $N$-complex $\Sigma^{-1}C(v)$
\begin{center}$0\rightarrow\bigoplus_{i=0}^\infty R\stackrel{v}\rightarrow \bigoplus_{i=0}^\infty R\stackrel{1}\rightarrow \bigoplus_{i=0}^\infty R\rightarrow\cdots\stackrel{1}\rightarrow \bigoplus_{i=0}^\infty R\rightarrow0$\end{center}concentrated in degrees $0,1,\cdots,N-1$. Given a sequence $\emph{\textbf{x}}=x_1,\cdots,x_d$ in $R$, we define
\begin{center}$\mathrm{Tel}(\emph{\textbf{x}};R):=\mathrm{Tel}(x_1;R)\otimes_R\cdots\otimes_R\mathrm{Tel}(x_d;R)$.\end{center}Then $\mathrm{Tel}(\emph{\textbf{x}};R)$ is an $N$-complex of free $R$-modules.}
\end{df}

\begin{lem}\label{lem:4.1}{\it{Let $\textbf{x}=x_1,\cdots,x_d$ be a sequence in $R$. One has a quasi-isomorphism
\begin{center}$w_\textbf{x}:\mathrm{Tel}(\textbf{x};R)\stackrel{\simeq}\rightarrow\check{C}(\textbf{x};R)$.\end{center}}}
\end{lem}
\begin{proof} For any $x_j$, by \cite[Lemma 5.7]{PSY}, one can define a quasi-isomorphism of $N$-complexes
\begin{center} $\xymatrix@C=23pt@R=20pt{
\mathrm{Tel}(x_j;R)\ar[d]^{w_{x_j}}:\ 0\ar[r]& \bigoplus_{i=0}^\infty R \ar[d]^{w_{x_j}^0}\ar[r]^v& \bigoplus_{i=0}^\infty R\ar[d]^{w_{x_j}^1}\ar@{=}[r] &\cdots\ar@{=}[r]&\bigoplus_{i=0}^\infty R\ar[d]^{w_{x_j}^{N-1}}\ar[r] &0 \\
\check{C}(x_j;R):\ 0\ar[r]& R \ar[r]^\iota& R_{x_j}\ar@{=}[r] &\cdots\ar@{=}[r]&R_{x_j}\ar[r] &0 }$
\end{center}where $w_{x_j}^1=\cdots=w_{x_j}^{N-1}$. Therefore, we have
 \begin{center}$\mathrm{Tel}(\emph{\textbf{x}};R)\cong\mathrm{Tel}(x_1;R)\otimes_R\cdots
 \otimes_R\mathrm{Tel}(x_d;R)\stackrel{\simeq}\rightarrow\check{C}(x_1;R)\otimes_R\cdots
 \otimes_R\check{C}(x_d;R)\cong\check{C}(\emph{\textbf{x}};R)$.
\end{center}This shows the quasi-isomorphism we seek.
\end{proof}

\begin{cor}\label{lem:4.4}{\it{Let $\textbf{x}$ and $\textbf{y}$ be two finite sequences in $R$ and let $\mathfrak{a}=(\textbf{x})$ and $\mathfrak{b}=(\textbf{y})$. If $\sqrt{\mathfrak{a}}=\sqrt{\mathfrak{b}}$, then $\mathrm{Tel}(\textbf{x};R)$ and $\mathrm{Tel}(\textbf{y};R)$ are homotopy equivalent.}}
\end{cor}
\begin{proof} By Lemma \ref{lem:4.1}, we have two quasi-isomorphisms $\mathrm{Tel}(\emph{\textbf{x}};R)\stackrel{\simeq}\rightarrow\check{C}(\emph{\textbf{x}};R)$ and
$\mathrm{Tel}(\emph{\textbf{y}};R)\stackrel{\simeq}\rightarrow\check{C}(\emph{\textbf{y}};R)$. But $\check{C}(\emph{\textbf{x}};R)\cong \check{C}(\emph{\textbf{y}};R)$ in $\mathbf{D}_N(R)$, it follows that $\mathrm{Tel}(\emph{\textbf{x}};R)\cong\mathrm{Tel}(\emph{\textbf{y}};R)$ in $\mathbf{D}_N(R)$. Consequently,   $\mathrm{Tel}(\emph{\textbf{x}};R)$ and $\mathrm{Tel}(\emph{\textbf{y}};R)$ are homotopy equivalent.
\end{proof}

\begin{cor}\label{lem:4.2}{\it{Let $\textbf{x}=x_1,\cdots,x_d$ be a weakly proregular sequence in $R$ and $\mathfrak{a}$ the ideal generated by $\textbf{x}$. For any $N$-complex $X$, there is a functorial quasi-isomorphism \begin{center}$\mathrm{R}\Gamma_{\mathfrak{a}}(X)\rightarrow \mathrm{Tel}(\textbf{x};R)\otimes_RX$.\end{center}}}
\end{cor}

Let $\mathfrak{a}$ be an ideal of $R$. We have an inverse system \begin{center}$\cdots\twoheadrightarrow R/\mathfrak{a}^{3}\twoheadrightarrow R/\mathfrak{a}^{2}\twoheadrightarrow R/\mathfrak{a}$.\end{center}Following \cite{GM}, for an $R$-module $M$ we write \begin{center}$\Lambda_\mathfrak{a}(M):=\underleftarrow{\textrm{lim}}_{s>0}(R/\mathfrak{a}^s\otimes_RM)$\end{center}
for the $\mathfrak{a}$-adic completion of $M$. We get an additive functor $\Lambda_\mathfrak{a}:\mathrm{Mod}R\rightarrow
\mathrm{Mod}R$ and there is a functorial morphism $\tau_M:M\rightarrow\Lambda_\mathfrak{a}(M)$ for any $M\in \mathrm{Mod}R$. By \cite[Corollary 3.29]{IKM}, the functor $\Lambda_\mathfrak{a}$  has a left derived functor
 \begin{center}$\mathrm{L}\Lambda_\mathfrak{a}(-):\mathbf{D}_N(R)\rightarrow\mathbf{D}_N(R),
 \xi':\mathrm{L}\Lambda_\mathfrak{a}\rightarrow\Lambda_\mathfrak{a}$\end{center}constructed using semi-projective resolutions. For any $N$-complex $X\in\mathbf{D}_N(R)$, by analogy with Proposition \ref{lem:2.0},  there is a functorial morphism $\tau^\mathrm{L}_X:X\rightarrow\mathrm{L}\Lambda_\mathfrak{a}(X)$ in $\mathbf{D}_N(R)$, such that $\xi'_X\circ\tau^\mathrm{L}_X=\tau_X$ as morphism $X\rightarrow\Lambda_\mathfrak{a}(X)$.

Let $x$ be an element of $R$ and $X$ an $N$-complex. Then the diagram (4.1) yields a morphism of inverse systems:
\begin{center} $\xymatrix@C=15pt@R=17pt{
\cdots \ar[r]& \mathrm{Hom}_R(\mathrm{Hom}_R(K^\bullet(x^2;R),R),X) \ar[d]\ar[r]& \mathrm{Hom}_R(\mathrm{Hom}_R(K^\bullet(x;R),R),X)\ar[d] \\
\cdots \ar[r]& R/(x^2)\otimes_RX\ar[r] & R/(x)\otimes_RX}$
\end{center}
This gives rise to a functorial morphism of $N$-complexes
\begin{center}$\lambda_X:\mathrm{Hom}_R(\mathrm{Tel}(x;R),X)\simeq
\mathrm{Hom}_R(\check{C}(x;R),X)
\rightarrow\underleftarrow{\textrm{lim}}_{s>0}(R/(x^s)\otimes_RY)=\Lambda_{(x)}(X)$.\end{center}

The next results provide an explicit formula for computing $\mathrm{L}\Lambda_\mathfrak{a}$.

\begin{thm}\label{lem:4.3}{\it{Let $\textbf{x}=x_1,\cdots,x_d$ be a weakly proregular sequence in $R$ and $\mathfrak{a}$ the ideal generated by $\textbf{x}$. For any $N$-complex $X$, there is a functorial quasi-isomorphism \begin{center}$\mathrm{Hom}_R(\mathrm{Tel}(\textbf{x};R),X)\stackrel{\simeq}\rightarrow
\mathrm{L}\Lambda_{\mathfrak{a}}(X)$.\end{center}}}
\end{thm}
\begin{proof} It is enough to consider a semi-projective $N$-complex $X=P$. By Lemma \ref{lem:3.3} we reduce to the case of a single projective module $P$. By \cite[Theorem 5.21]{PSY}, one can obtain a quasi-isomorphism $\mathrm{Hom}_R(\mathrm{Tel}(x_1;R),P)\stackrel{\simeq}\rightarrow \Lambda_{(x_1)}(P)$ in $\mathbf{D}_N(R)$, where $\Lambda_{(x_1)}(P)$ is viewed as the $N$-complex $D^0_1(\Lambda_{(x_1)}(P))$. By induction, we obtain the quasi-isomorphism we seek.
\end{proof}

\begin{cor}\label{lem:4.4}{\it{Let $\mathfrak{a}$ be a weakly proregular ideal of $R$.

$\mathrm{(1)}$ For any $N$-complex $X$, there exists a functorial quasi-isomorphism \begin{center}$\mathrm{RHom}_R(\mathrm{R}\Gamma_{\mathfrak{a}}(R),X)\stackrel{\simeq}\rightarrow
\mathrm{L}\Lambda_{\mathfrak{a}}(X)$.\end{center}

$\mathrm{(2)}$ The morphism
$\tau^\mathrm{L}_{\mathrm{L}\Lambda_\mathfrak{a}(X)}:
\mathrm{L}\Lambda_\mathfrak{a}(X)\rightarrow\mathrm{L}\Lambda_\mathfrak{a}(\mathrm{L}\Lambda_\mathfrak{a}(X))$
is an isomorphism. Thus the functor
$\mathrm{L}\Lambda_\mathfrak{a}:\mathbf{D}_N(R)\rightarrow \mathbf{D}_N(R)$
is idempotent.}}
\end{cor}

\section{\bf MGM equivalence of $N$-complexes}
The task of this section to prove the MGM equivalence in $\mathbf{D}_N(R)$, i.e.,
we show an equivalence between the category of cohomologically $\mathfrak{a}$-torsion $N$-complexes and the category of cohomologically $\mathfrak{a}$-adic complete $N$-complexes.

\begin{df}\label{lem:5.1} {\rm (1) An $N$-complex $X\in\mathbf{D}_N(R)$ is called cohomologically $\mathfrak{a}$-torsion if the morphism
$\sigma^\mathrm{R}_X:\mathrm{R}\Gamma_{\mathfrak{a}}(X)\rightarrow X$ is an isomorphism.
 The full subcategory of $\mathbf{D}_N(R)$ consisting of cohomologically $\mathfrak{a}$-torsion
$N$-complexes is denoted by $\mathbf{D}_N(R)_{\mathfrak{a}\textrm{-tor}}$.

(2) An $N$-complex $Y\in\mathbf{D}_N(R)$ is called cohomologically $\mathfrak{a}$-adic complete if the
morphism $\tau^\mathrm{L}_Y:Y\rightarrow\mathrm{L}\Lambda_{\mathfrak{a}}(Y)$ is an isomorphism.
 The full subcategory of $\mathbf{D}_N(R)$ consisting of cohomologically $\mathfrak{a}$-adic
complete $N$-complexes is denoted by $\mathbf{D}_N(R)_{\mathfrak{a}\textrm{-com}}$.}
\end{df}

We first show that the functor $\mathrm{R}\Gamma_{\mathfrak{a}}$
is right adjoint to the inclusion $\mathbf{D}_N(R)_{\mathfrak{a}\textrm{-tor}}\hookrightarrow\mathbf{D}_N(R)$ and the functor $\mathrm{L}\Lambda_{\mathfrak{a}}$
is left adjoint to the inclusion $\mathbf{D}_N(R)_{\mathfrak{a}\textrm{-com}}\hookrightarrow\mathbf{D}_N(R)$.

\begin{prop}\label{lem:5.3}{\it{$\mathrm{(1)}$ The morphism $\sigma^\mathrm{R}_Y:\mathrm{R}\Gamma_{\mathfrak{a}}(Y)\rightarrow Y$ induces an isomorphism\begin{center}$\mathrm{Hom}_{\mathbf{D}_N(R)_{\mathfrak{a}\textrm{-}\mathrm{tor}}}(X,
\mathrm{R}\Gamma_{\mathfrak{a}}(Y))\stackrel{\simeq}\longrightarrow\mathrm{Hom}_{\mathbf{D}_N(R)}(X,Y),\ \forall\ X\in\mathbf{D}_N(R)_{\mathfrak{a}\textrm{-}\mathrm{tor}}, Y\in\mathbf{D}_N(R)$.\end{center}

$\mathrm{(2)}$ The morphism $\tau^\mathrm{L}_X:X\rightarrow \mathrm{L}\Lambda_{\mathfrak{a}}(X)$ induces an isomorphism\begin{center}$\mathrm{Hom}_{\mathbf{D}_N(R)_{\mathfrak{a}\textrm{-}\mathrm{com}}}
(\mathrm{L}\Lambda_{\mathfrak{a}}(X),Y)
\stackrel{\simeq}\longrightarrow\mathrm{Hom}_{\mathbf{D}_N(R)}(X,Y),\ \forall\ X\in\mathbf{D}_N(R), Y\in\mathbf{D}_N(R)_{\mathfrak{a}\textrm{-}\mathrm{com}}$.\end{center}}}
\end{prop}
\begin{proof} We just prove (1) since (2) follows by duality.

 We need to show that $\varrho_{X,Y}:\mathrm{Hom}_{\mathbf{D}_N(R)_{\mathfrak{a}\textrm{-tor}}}(X,\mathrm{R}\Gamma_{\mathfrak{a}}(Y))
=\mathrm{Hom}_{\mathbf{D}_N(R)}(X,\mathrm{R}\Gamma_{\mathfrak{a}}(Y))
\rightarrow\mathrm{Hom}_{\mathbf{D}_N(R)}(X,Y)$ is an isomorphism. Referring then to the diagram \begin{center}$\mathrm{Hom}_{\mathbf{D}_N(R)}(X,Y)\stackrel{\nu}\longrightarrow
\mathrm{Hom}_{\mathbf{D}_N(R)}(\mathrm{R}\Gamma_{\mathfrak{a}}(X),\mathrm{R}\Gamma_{\mathfrak{a}}(Y))
\stackrel{\rho}\longleftarrow\mathrm{Hom}_{\mathbf{D}_N(R)}(X,\mathrm{R}\Gamma_{\mathfrak{a}}(Y))$,\end{center}
where $\nu$ is the natural morphism and $\rho$ is induced by the isomorphism
$\sigma^\mathrm{R}_X:\mathrm{R}\Gamma_{\mathfrak{a}}(X)\rightarrow X$. Next we show that $\rho^{-1}\nu$ is inverse to $\varrho_{X,Y}$.
That $\varrho_{X,Y}\rho^{-1}\nu(\alpha)=\alpha$ for any $\alpha\in\mathrm{Hom}_{\mathbf{D}_N(R)}(X,Y)$ amounts to the (obvious)
commutativity of the diagram
\begin{center} $\xymatrix@C=35pt@R=20pt{
\mathrm{R}\Gamma_{\mathfrak{a}}(X)\ar[d]_{\sigma^\mathrm{R}_X}^\simeq\ar[r]^{\mathrm{R}\Gamma_{\mathfrak{a}}(\alpha)} &\mathrm{R}\Gamma_{\mathfrak{a}}(Y)\ar[d]^{\sigma^\mathrm{R}_Y}\\
X\ar[r]^{\alpha}& Y.}$
\end{center}
That $\rho^{-1}\nu\varrho_{X,Y}(\beta)=\beta$ for $\beta\in\mathrm{Hom}_{\mathbf{D}_N(R)}(X,\mathrm{R}\Gamma_{\mathfrak{a}}(Y))$ amounts to commutativity of
\begin{center} $\xymatrix@C=35pt@R=20pt{
\mathrm{R}\Gamma_{\mathfrak{a}}(X)\ar[d]_{\sigma^\mathrm{R}_X}^\simeq\ar[r]^{\mathrm{R}\Gamma_{\mathfrak{a}}(\beta)\ \ \ \ } &\mathrm{R}\Gamma_{\mathfrak{a}}(\mathrm{R}\Gamma_{\mathfrak{a}}(Y))\ar[d]^\simeq\\
X\ar[r]^{\beta\ \ \ }& \mathrm{R}\Gamma_{\mathfrak{a}}(Y).}$
\end{center}This shows the isomorphism we seek.
\end{proof}

Here is the main result of our paper, similar results can be found in \cite[Section 7]{PSY}.

\begin{thm}\label{lem:5.4}{\it{Let $\mathfrak{a}$ be a weakly proregular ideal of $R$.

$\mathrm{(1)}$ For any $X\in\mathbf{D}_N(R)$, the morphism $\mathrm{L}\Lambda_{\mathfrak{a}}(\sigma^\mathrm{R}_X):\mathrm{L}\Lambda_{\mathfrak{a}}
(\mathrm{R}\Gamma_{\mathfrak{a}}(X))\rightarrow\mathrm{L}\Lambda_{\mathfrak{a}}(X)$ is an isomorphism.

$\mathrm{(2)}$ For any $X\in\mathbf{D}_N(R)$, the morphism $\mathrm{R}\Gamma_{\mathfrak{a}}(\tau^\mathrm{L}_X):\mathrm{R}\Gamma_{\mathfrak{a}}(X)
\rightarrow\mathrm{R}\Gamma_{\mathfrak{a}}(\mathrm{L}\Lambda_{\mathfrak{a}}(X))$ is an isomorphism.

$\mathrm{(3)}$ For any $X\in\mathbf{D}_N(R)$, one has $\mathrm{R}\Gamma_{\mathfrak{a}}(X)\in\mathbf{D}_N(R)_{\mathfrak{a}\textrm{-}\mathrm{tor}}$ and $\mathrm{L}\Lambda_{\mathfrak{a}}(X)\in\mathbf{D}_N(R)_{\mathfrak{a}\textrm{-}\mathrm{com}}$.

$\mathrm{(4)}$ The functors $\mathrm{R}\Gamma_{\mathfrak{a}}:\mathbf{D}_N(R)_{\mathfrak{a}\textrm{-}\mathrm{com}}\rightleftarrows\mathbf{D}_N(R)_{\mathfrak{a}\textrm{-}\mathrm{tor}}:\mathrm{L}\Lambda_{\mathfrak{a}}$ form an equivalence.}}
\end{thm}
\begin{proof} (1) By Corollary \ref{lem:4.2} and Theorem \ref{lem:4.3}, we have a commutative diagram\begin{center} $\xymatrix@C=90pt@R=20pt{
\mathrm{L}\Lambda_{\mathfrak{a}}(\mathrm{R}\Gamma_{\mathfrak{a}}(X))\ar[d]_{\simeq}\ar[r]^{\ \ \ \ \mathrm{L}\Lambda_{\mathfrak{a}}(\sigma^\mathrm{R}_X)} &\mathrm{L}\Lambda_{\mathfrak{a}}(X)\ar[d]^\simeq\\
\mathrm{RHom}_{R}(\mathrm{R}\Gamma_{\mathfrak{a}}(R),\mathrm{R}\Gamma_{\mathfrak{a}}(X))\ar[r]^{\ \ \ \ \mathrm{RHom}_{R}(\mathrm{R}\Gamma_{\mathfrak{a}}(R),\sigma^\mathrm{R}_X)}& \mathrm{RHom}_{R}(\mathrm{R}\Gamma_{\mathfrak{a}}(R),X).}$
\end{center}Since $\mathrm{R}\Gamma_{\mathfrak{a}}(R)\in\mathbf{D}_N(R)_{\mathfrak{a}\textrm{-}\mathrm{tor}}$, it follows from
Proposition \ref{lem:5.3}(1) that $\mathrm{RHom}_{R}(\mathrm{R}\Gamma_{\mathfrak{a}}(R),\sigma^\mathrm{R}_X)$ is an isomorphism, so is  $\mathrm{L}\Lambda_{\mathfrak{a}}(\sigma^\mathrm{R}_X)$.

(2) By Corollary \ref{lem:4.2} and Theorem \ref{lem:4.3}, one has a commutative diagram\begin{center} $\xymatrix@C=85pt@R=20pt{
\mathrm{R}\Gamma_{\mathfrak{a}}(R)\otimes^\mathrm{L}_RX\ar[d]_{\simeq}
\ar[r]^{\mathrm{R}\Gamma_{\mathfrak{a}}(R)\otimes^\mathrm{L}_R\tau^\mathrm{L}_X\ \ \ \ \  } &\mathrm{R}\Gamma_{\mathfrak{a}}(R)\otimes^\mathrm{L}_R\mathrm{L}\Lambda_{\mathfrak{a}}(X)\ar[d]^\simeq\\
\mathrm{R}\Gamma_{\mathfrak{a}}(X)\ar[r]^{\mathrm{R}\Gamma_{\mathfrak{a}}(\tau^\mathrm{L}_X)\ \ }& \mathrm{R}\Gamma_{\mathfrak{a}}(\mathrm{L}\Lambda_{\mathfrak{a}}(X)).}$
\end{center}Let $E$ be a faithful injective $R$-module. Applying the functor $\mathrm{RHom}_R(-,E)$ to $\mathrm{R}\Gamma_{\mathfrak{a}}(R)\otimes^\mathrm{L}_R\tau^\mathrm{L}_X$, we have a commutative diagram:
\begin{center} $\xymatrix@C=110pt@R=20pt{
\mathrm{RHom}_{R}(\mathrm{R}\Gamma_{\mathfrak{a}}(R)\otimes^\mathrm{L}_R\mathrm{L}\Lambda_{\mathfrak{a}}(X),E)
\ar[d]_{\simeq}\ar[r]^{ \ \ \ \ \mathrm{RHom}_{R}(\mathrm{R}\Gamma_{\mathfrak{a}}(R)\otimes^\mathrm{L}_R\tau^\mathrm{L}_X,E)} &\mathrm{RHom}_{R}(\mathrm{R}\Gamma_{\mathfrak{a}}(R)\otimes^\mathrm{L}_RX,E)\ar[d]^\simeq\\
\mathrm{RHom}_{R}(\mathrm{L}\Lambda_{\mathfrak{a}}(X),\mathrm{L}\Lambda_{\mathfrak{a}}(E))\ar[r]^{\ \ \ \ \ \mathrm{RHom}_{R}(\tau^\mathrm{L}_X,\mathrm{L}\Lambda_{\mathfrak{a}}(E))}& \mathrm{RHom}_{R}(X,\mathrm{L}\Lambda_{\mathfrak{a}}(E)).}$
\end{center}Since $\mathrm{L}\Lambda_{\mathfrak{a}}(X)\in\mathbf{D}_N(R)_{\mathfrak{a}\textrm{-}\mathrm{com}}$, it follows from Proposition \ref{lem:5.3}(2) that $\mathrm{RHom}_{R}(\tau^\mathrm{L}_X,\mathrm{L}\Lambda_{\mathfrak{a}}(E))$ is an isomorphism, so is  $\mathrm{RHom}_{R}(\mathrm{R}\Gamma_{\mathfrak{a}}(R)\otimes^\mathrm{L}_R\tau^\mathrm{L}_X,E)=
\mathrm{Hom}_{R}(\mathrm{R}\Gamma_{\mathfrak{a}}(R)\otimes^\mathrm{L}_R\tau^\mathrm{L}_X,E)$.
But $E$ is faithful injective, so $\mathrm{R}\Gamma_{\mathfrak{a}}(R)\otimes^\mathrm{L}_R\tau^\mathrm{L}_X$ is an isomorphism, and hence $\mathrm{R}\Gamma_{\mathfrak{a}}(\tau^\mathrm{L}_X)$ is an isomorphism.

(3) This is immediate from the idempotence of the functors $\mathrm{R}\Gamma_{\mathfrak{a}}$ and $\mathrm{L}\Lambda_{\mathfrak{a}}$.

(4) By (1), there are functorial isomorphisms \begin{center}$X\cong\mathrm{L}\Lambda_{\mathfrak{a}}(X)\cong
\mathrm{L}\Lambda_{\mathfrak{a}}(\mathrm{R}\Gamma_{\mathfrak{a}}(X))$ for $X\in\mathbf{D}_N(R)_{\mathfrak{a}\textrm{-}\mathrm{com}}$.\end{center}By (2), there are functorial isomorphisms \begin{center}$X\cong\mathrm{R}\Gamma_{\mathfrak{a}}(X)\cong
\mathrm{R}\Gamma_{\mathfrak{a}}(\mathrm{L}\Lambda_{\mathfrak{a}}(X))$ for $X\in\mathbf{D}_N(R)_{(\mathfrak{a})\textrm{-}\mathrm{tor}}$.\end{center} These isomorphisms yield the desired
equivalence.
\end{proof}

\begin{rem}\label{lem:5.5}{\rm Let $\mathfrak{a}$ be a weakly proregular ideal of $R$. For any $X,Y\in\mathbf{D}_N(R)$, one has that the morphisms
\begin{center}$\begin{aligned}\mathrm{RHom}_{R}(\mathrm{R}\Gamma_{\mathfrak{a}}(X),\mathrm{R}\Gamma_{\mathfrak{a}}(Y))
& \xleftarrow{\mathrm{RHom}_{R}(\mathrm{R}\Gamma_{\mathfrak{a}}(\tau^\mathrm{L}_X),1)}
\mathrm{RHom}_{R}(\mathrm{R}\Gamma_{\mathfrak{a}}(\mathrm{L}\Lambda_{\mathfrak{a}}(X)),
\mathrm{R}\Gamma_{\mathfrak{a}}(Y))\\
& \xrightarrow{\textrm{adjunction}}\mathrm{RHom}_{R}(\mathrm{L}\Lambda_{\mathfrak{a}}(X),
\mathrm{L}\Lambda_{\mathfrak{a}}(\mathrm{R}\Gamma_{\mathfrak{a}}(Y)))\\
& \xrightarrow{\mathrm{RHom}_{R}(1,\mathrm{L}\Lambda_{\mathfrak{a}}(\sigma^\mathrm{R}_Y))}
\mathrm{RHom}_{R}(\mathrm{L}\Lambda_{\mathfrak{a}}(X),\mathrm{L}\Lambda_{\mathfrak{a}}(Y))\end{aligned}$\end{center}is an isomorphism in $\mathbf{D}_N(R)$.}
\end{rem}

\section{\bf Invariant}

In this section, we prove that over a commutative noetherian ring, via Koszul cohomology, via RHom cohomology (resp. $\otimes$ cohomology) and via local cohomology (resp. derived completion), all yield the same invariant.

\begin{lem}\label{lem:6.1}{\it{Let $\textbf{x}=x_1,... ,x_d$ be a sequence of $R$ and $X$ an $N$-complex. For $t=1,\cdots,N-1$, one has $(\textbf{x})\mathrm{H}_t(\textbf{x};X)=0=(\textbf{x})\mathrm{H}_t(\mathrm{Hom}_R(K^\bullet(\textbf{x};R);X))$.}}
\end{lem}
\begin{proof} For each $x_i$, the morphism of $N$-complexes
$K^\bullet(x_i;R)\rightarrow K^\bullet(x_i;R)$ given by multiplication by $x_i$ can be factored as follows:
\begin{center}
$\xymatrix@C=20pt@R=18pt{
   K^\bullet(x_i;R):\ 0\ar[r] & R\ar@{=}[r] \ar[d]^{x_i}& R\ar@{=}[r] \ar[d]^{x_i}&\cdots\ar@{=}[r]& R\ar@{=}[r] \ar[d]^{x_i}& R\ar[r]^{x_i} \ar[d]^{x_i}& R\ar[r] \ar@{=}[d]& 0\\
  D^0_N(R):\ 0\ar[r] & R\ar@{=}[r] \ar@{=}[d]& R\ar@{=}[r] \ar@{=}[d]&\cdots\ar@{=}[r]& R\ar@{=}[r] \ar@{=}[d]& M\ar@{=}[r] \ar@{=}[d]& M\ar[r] \ar[d]^{x_i}& 0\\
K^\bullet(x_i;R):\ 0\ar[r] & R\ar@{=}[r] & R\ar@{=}[r] &\cdots\ar@{=}[r]& R\ar@{=}[r] & R\ar[r]^{x_i}&R\ar[r] & 0}$
\end{center}Thus multiplication by $x_i$ is null-homotopic on $K^\bullet(x_i;R)$, so the same hold for $K^\bullet(\emph{\textbf{x}};R)\otimes_RX$ and $\mathrm{Hom}_R(K^\bullet(\emph{\textbf{x}};R);X)$. Therefore, $x_i\mathrm{H}_t(\emph{\textbf{x}};X)=0=x_i\mathrm{H}_t(\mathrm{Hom}_R(K^\bullet(\emph{\textbf{x}};R);X)$ for each $x_i$, and
hence for $\emph{\textbf{x}}=x_1,... ,x_d$, as desired.
\end{proof}

\begin{lem}\label{lem:6.2}{\it{Let $\textbf{x}=x_1,... ,x_d$ be a sequence of $R$ and $\mathfrak{a}$ the ideal generated by $\textbf{x}$, and let $X$ be an $N$-complex. For $t=1,\cdots,N-1$, one has
\begin{center}$\mathrm{Hom}_R(K^\bullet(\textbf{x};R),X)\cong
\mathrm{R}\Gamma_\mathfrak{a}(\mathrm{Hom}_R(K^\bullet(\textbf{x};R),X))$.\end{center}}}
\end{lem}
\begin{proof} By the definition of $\mathfrak{a}$-torsion functor, we have the following isomorphisms \begin{center}$\begin{aligned}\mathrm{R}\Gamma_\mathfrak{a}(\mathrm{Hom}_R(K^\bullet(\emph{\textbf{x}};R),X))
&\cong\underrightarrow{\textrm{lim}}\mathrm{RHom}_R(R/\mathfrak{a}^s,\mathrm{Hom}_R(K^\bullet(\emph{\textbf{x}};R),X))\\
&\cong\underrightarrow{\textrm{lim}}\mathrm{RHom}_R(R/\mathfrak{a}^s\otimes_RK^\bullet(\emph{\textbf{x}};R),X)\\
&\cong\mathrm{Hom}_R(K^\bullet(\emph{\textbf{x}};R),\underrightarrow{\textrm{lim}}\mathrm{RHom}_R(R/\mathfrak{a}^s,X))\\
&\cong\mathrm{Hom}_R(K^\bullet(\emph{\textbf{x}};R),\mathrm{R}\Gamma_\mathfrak{a}(X))\\
&\cong\mathrm{Hom}_R(K^\bullet(\emph{\textbf{x}};R),X),\end{aligned}$\end{center}
where the fifth one is by $K^\bullet(\emph{\textbf{x}};R)\in\mathbf{D}_N(R)_{\mathfrak{a}\textrm{-}\mathrm{tor}}$ and Proposition \ref{lem:5.3}(1).
\end{proof}

\begin{lem}\label{lem:3.00}{\it{Let $x$ be an element in $R$. For $i\in\mathbb{Z}$ and a fixed $t$, one has
\begin{center}$\mathrm{H}^i_t(\mathrm{Hom}_{R}(K(x;R),X))=0\ \textrm{implies}\ \mathrm{H}^i_t(\mathrm{Hom}_R(K(x^s),X))=0\ \textrm{for}\ s>0$.\end{center}}}
\end{lem}
\begin{proof} By Octahedral axiom, we have a commutative diagram in $\mathbf{K}_N(R)$\begin{center}$\xymatrix@C=20pt@R=20pt{
  0\ar[d]\ar[r]& K(x;R) \ar[d]\ar@{=}[r]& K(x;R)\ar[d]\ar[r]&0\ar[d]\\
   R\ar@{=}[d]\ar[r] & K(x^2;R)\ar[d]\ar[r]& \Sigma R\ar[d]^x\ar[r]^{x^2} &\Sigma R \ar@{=}[d]\\
  R\ar[d]\ar[r] & K(x;R)\ar[d]\ar[r]& \Sigma R\ar[d]\ar[r]^{x} &\Sigma R\ar[d]\\
   0\ar[r]& \Sigma K(x;R) \ar@{=}[r]&\Sigma K(x;R)\ar[r]&0}$\end{center}where all rows and columns are exact triangles in $\mathbf{K}_N(R)$. Applying the functor $\mathrm{RHom}_R(-,X)$ to the second column, one gets an exact triangle \begin{center}$\Sigma^{-1}\mathrm{Hom}_R(K(x;R),X)\rightarrow\mathrm{Hom}_R(K(x;R),X)
   \rightarrow\mathrm{Hom}_R(K(x^2;R),X)\rightarrow \mathrm{Hom}_R(K(x;R),X)$,\end{center} which implies that $\mathrm{H}^i_t(\mathrm{Hom}_R(K(x^2;R),X))=0$ whenever
$\mathrm{H}^i_t(\mathrm{Hom}_{R}(K(x;R),X))=0$. By repeating this process we get the claim.
\end{proof}

For an $N$-complex $X$, set \begin{center}$\sigma_{\leq n}X:\ \cdots\xrightarrow{d^{n-N}}X^{n-N+1}\xrightarrow{d^{N-N+1}}\mathrm{Z}^{n-N+2}_{N-1}(X)
\xrightarrow{d^{n-N+2}}\cdots\xrightarrow{d^{n+1}}\mathrm{Z}^{n}_1(X)\rightarrow0$,\end{center}
\begin{center}$\sigma_{\geq n}X:\ 0\rightarrow\textrm{C}^{n}_{N-1}(X)\xrightarrow{\bar{d}^{n}}\cdots\xrightarrow{\bar{d}^{N+N-3}} \textrm{C}^{n}_1(X)\xrightarrow{\bar{d}^{N+N-2}} X^{n+N-1}\xrightarrow{d^{N+N-1}}\cdots$.\end{center}

The next result shows that Koszul cohomology, RHom cohomology and local cohomology yield the same invariant, which was proved by Foxby and Iyengar \cite{FI} for $N=2$ (see \cite[Theorem 2.1]{FI}).

\begin{thm}\label{lem:6.3}{\it{Let $\mathfrak{a}$ be an ideal of a noetherian ring $R$ and $K$ the Koszul
$N$-complex on a sequence of $n$ generators for $\mathfrak{a}$. For any $X\in\mathbf{D}_N(R)$ and a fixed $t$, one has
\begin{center}$\begin{aligned}\mathrm{inf}\{\ell\in\mathbb{Z}\hspace{0.03cm}|\hspace{0.03cm}\mathrm{H}^\ell_t
(\mathrm{RHom}_R(R/\mathfrak{a},X))\neq0\}
&=\mathrm{inf}\{\ell\in\mathbb{Z}\hspace{0.03cm}|\hspace{0.03cm}\mathrm{H}^\ell_{t,\mathfrak{a}}(X)\neq0\}\\
&=\mathrm{inf}\{\ell\in\mathbb{Z}\hspace{0.03cm}|\hspace{0.03cm}\mathrm{H}^\ell_t
(\mathrm{Hom}_R(K,X))\neq0\}.\end{aligned}$\end{center}}}
\end{thm}
\begin{proof} Denote the three numbers in question $a,b,c$, respectively.

For an $R/\mathfrak{a}$-module $T$, one can set $P\stackrel{\nu}\rightarrow T$ be a semi-projective resolution such that $P^i=0$ for all $i>0$ by \cite[Proposition 3.4]{YW}. Hence
\begin{center}$\begin{aligned}\mathrm{RHom}_R(T,X)
&\cong\mathrm{Hom}_R(P,X)\\
&\cong\mathrm{Hom}_{R/\mathfrak{a}}(P,\mathrm{RHom}_R(R/\mathfrak{a},X))\\
&\cong\mathrm{Hom}_{R/\mathfrak{a}}(P,\sigma_{\geq a}(\mathrm{RHom}_R(R/\mathfrak{a},X))),\end{aligned}$\end{center}
where the last isomorphism is by the dual of \cite[Lemma 3.9]{IKM}. For $n<a$ and $i\in\mathbb{Z}$, one of the inequalities $i>0$ or $n+i<a$ holds, so $\mathrm{Hom}_R(P,X)^n=
\prod_{i\in\mathbb{Z}}\mathrm{Hom}_R(P^i,X^{n+i})=0$. So
\begin{align}
\mathrm{H}^\ell_t(\mathrm{RHom}_R(T,X))=0\ \textrm{for}\ \ell<a.
\label{exact03}
\tag{\dag}\end{align}
Apply the functor $\mathrm{RHom}_R(-,X)$ to the exact triangle
$\mathfrak{a}^s/\mathfrak{a}^{s+1}\rightarrow R/\mathfrak{a}^{s+1}\rightarrow R/\mathfrak{a}^{s}\rightarrow\Sigma\mathfrak{a}^s/\mathfrak{a}^{s+1}$ in $\mathbf{D}_N(A)$
yields the long exact sequence
\begin{center}$\cdots\rightarrow\mathrm{H}^\ell_t(\mathrm{RHom}_R(R/\mathfrak{a}^{s},X))\rightarrow \mathrm{H}^\ell_t(\mathrm{RHom}_R(R/\mathfrak{a}^{s+1},X))\rightarrow \mathrm{H}^\ell_t(\mathrm{RHom}_R(\mathfrak{a}^{s}/\mathfrak{a}^{s+1},X))\rightarrow\cdots$.\end{center}
Then $(\dag)$ implies that $\mathrm{H}^\ell_t(\mathrm{RHom}_R(\mathfrak{a}^{s}/\mathfrak{a}^{s+1},X))=0$ for $\ell<a$.
 By the induction hypothesis and
the long exact sequence above, we get that \begin{center}$\mathrm{H}^\ell_t(\mathrm{RHom}_R(R/\mathfrak{a}^{s+1},X))=0$ for $\ell<a$ and $s>0$. \end{center} Hence $\mathrm{H}^\ell_t(\mathrm{R}\Gamma_\mathfrak{a}(X))=0$ for $\ell<a$. On the other hand, let $X\stackrel{\simeq}\rightarrow I$ be a semi-injective resolution. Then the inverse system of epimorphisms $\cdots\twoheadrightarrow R/\mathfrak{a}^3\twoheadrightarrow R/\mathfrak{a}^2\twoheadrightarrow R/\mathfrak{a}$ induces a direct system of monomorphisms \begin{center}$\mathrm{Hom}_R(R/\mathfrak{a},I)\rightarrowtail\mathrm{Hom}_R(R/\mathfrak{a}^2,I)
\rightarrowtail\mathrm{Hom}_R(R/\mathfrak{a}^3,I)\rightarrowtail\cdots$\end{center}
So $\mathrm{H}^\ell_t(\mathrm{RHom}_R(R/\mathfrak{a},X))=0$ for $\ell<b$. This shows that $a=b$.

By Lemma \ref{lem:6.2} and construction of $K$, one has that
$\mathrm{Hom}_R(K,X)\simeq\mathrm{R}\Gamma_\mathfrak{a}(\mathrm{Hom}_R(K,X))\simeq
\mathrm{Hom}_R(K,\mathrm{R}\Gamma_\mathfrak{a}(X))$.
Hence we get
$\mathrm{H}^\ell_t(\mathrm{Hom}_R(K,X))=0$ for $\ell<b$.
On the other hand, one has that
$\mathrm{H}^\ell_{t,\mathfrak{a}}(X)=0$ for $\ell<c$ by Lemma \ref{lem:3.00}. This shows the equality $b=c$.
\end{proof}

The next result shows that Koszul cohomology, $\otimes$ cohomology and derived completion yield the same invariant, which was proved by Foxby and Iyengar \cite{FI} for $N=2$ (see \cite[Theorem 4.1]{FI}).

\begin{thm}\label{lem:6.6}{\it{Let $\mathfrak{a}$ be an ideal of a noetherian ring $R$ and $K$ the Koszul
$N$-complex on a sequence of $n$ generators for $\mathfrak{a}$. For any $X\in\mathbf{D}_N(R)$ and a fixed $t$, one has
\begin{center}$\begin{aligned}\mathrm{sup}\{\ell\hspace{0.03cm}|\hspace{0.03cm}\mathrm{H}^\ell_t
(R/\mathfrak{a}\otimes^\mathrm{L}_RX)\neq0\}
&=\mathrm{sup}\{\ell\hspace{0.03cm}|\hspace{0.03cm}
\mathrm{H}^\ell_t(\mathrm{L}\Lambda_\mathfrak{a}(X))\neq0\}\\
&=\mathrm{sup}\{\ell\hspace{0.03cm}|\hspace{0.03cm}\mathrm{H}^\ell_t
(K\otimes_RX)\neq0\}.\end{aligned}$\end{center}}}
\end{thm}
\begin{proof} Denote the three numbers in question $a,b,c$, respectively.

Let $P\stackrel{\simeq}\rightarrow X$ be a semi-projective resolution. Then the inverse system of epimorphisms $\cdots\twoheadrightarrow R/\mathfrak{a}^3\twoheadrightarrow R/\mathfrak{a}^2\twoheadrightarrow R/\mathfrak{a}$ induces an inverse system of epimorphisms \begin{center}$\cdots\twoheadrightarrow R/\mathfrak{a}^3\otimes_RP\twoheadrightarrow R/\mathfrak{a}^2\otimes_RP\twoheadrightarrow R/\mathfrak{a}\otimes_RP$.\end{center}
So $\mathrm{H}^\ell_t(R/\mathfrak{a}\otimes^\mathrm{L}_RX)=0$ for $\ell>b$. To the opposite inequality, note that
$T\otimes^\mathrm{L}_RX\cong T\otimes^\mathrm{L}_{R/\mathfrak{a}}R/\mathfrak{a}\otimes^\mathrm{L}_RX$ for any $R/\mathfrak{a}$-module $T$. By analogy with the proof of Theorem \ref{lem:6.3}, one has
\begin{align}
\mathrm{H}^\ell_t
(T\otimes^\mathrm{L}_RX)=0\ \textrm{for}\ \ell>a.
\label{exact03}
\tag{\ddag}\end{align}
Apply $-\otimes^\mathrm{L}_RX$ to the exact triangle $\mathfrak{a}^s/\mathfrak{a}^{s+1}\rightarrow R/\mathfrak{a}^{s+1}\rightarrow R/\mathfrak{a}^{s}\rightarrow\Sigma\mathfrak{a}^s/\mathfrak{a}^{s+1}$ in $\mathbf{D}_N(R)$
yields the following long exact sequence
\begin{center}$\cdots\rightarrow\mathrm{H}^\ell_t(\mathfrak{a}^{s}/\mathfrak{a}^{s+1}\otimes^\mathrm{L}_RX)
\rightarrow\mathrm{H}^\ell_t(R/\mathfrak{a}^{s+1}\otimes^\mathrm{L}_RX)\rightarrow \mathrm{H}^\ell_t(R/\mathfrak{a}^{s}\otimes^\mathrm{L}_RX)\rightarrow\cdots$.\end{center}
Then $(\ddag)$ yields that $\mathrm{H}^\ell_t(\mathfrak{a}^{s}/\mathfrak{a}^{s+1}\otimes^\mathrm{L}_RX)=0$ for $\ell>a$.
 By the induction hypothesis and
the long exact sequence above, we get that \begin{center}$\mathrm{H}^\ell_t(R/\mathfrak{a}^{s+1}\otimes^\mathrm{L}_RX)=0\ \textrm{for}\ \ell>a\ \textrm{and}\ s>0$. \end{center} This implies that $a=b$.

Let $E$ be a faithful injective $R$-module. We have the following equivalences
\begin{center}$\begin{aligned}\mathrm{H}^\ell_t(K\otimes_RX)=0\ \textrm{for}\ \ell>a
&\Longleftrightarrow\mathrm{H}^{-\ell}_{N-t}(\mathrm{Hom}_R(K,\mathrm{Hom}_R(X,E)))=0\ \textrm{for}\ \ell>a\\
&\Longleftrightarrow\mathrm{H}^{-\ell}_{N-t}(\mathrm{RHom}_R(R/\mathfrak{a},\mathrm{Hom}_R(X,E)))=0\ \textrm{for}\ \ell>a\\
&\Longleftrightarrow\mathrm{H}^\ell_t(R/\mathfrak{a}\otimes_RX)=0\ \textrm{for}\ \ell>a,\end{aligned}$\end{center}where the third equivalence is by
Theorem \ref{lem:6.3}. This shows the equality $a=c$.
\end{proof}

\bigskip \centerline {\bf ACKNOWLEDGEMENTS}
\bigskip This research was partially supported by National Natural Science Foundation of China (11761060).

\end{document}